\documentclass[11pt]{amsart}

\usepackage{amscd}
\usepackage{amsmath}
\usepackage{amssymb}
\usepackage{amsthm}
\usepackage{epsf}
\usepackage{latexsym}
\usepackage{verbatim}
\usepackage[all, cmtip]{xy}
\usepackage{tikz}
\usepackage[margin=1.0in]{geometry}

\newtheorem{theorem}{Theorem}[section]

\newtheorem{lemma}[theorem]{Lemma}

\newtheorem{corollary}[theorem]{Corollary}
\newtheorem{proposition}[theorem]{Proposition}

\newtheorem{varexample}[theorem]{Example}
\theoremstyle{definition}
\newtheorem{definition}[theorem]{Definition}
 \newtheorem*{ack}{Acknowledgements}
\newtheorem{remark}[theorem]{Remark}

\newcommand{\Spec}{\mathrm{Spec}\,}
\newcommand{\PP}{\mathbb{P} }

\newcommand{\Ga}{\mathbb{G}_a}

\newcommand{\mult}{\operatorname{mult}}

\newcommand{\Eff}{\overline{\operatorname{Eff}}}
\newcommand{\HH}{\operatorname{H}}

\newcommand{\M}{\overline{M}}
\newcommand{\Pic}{\operatorname{Pic}}
\newcommand{\Hom}{\operatorname{Hom}}

\newcommand{\Cox}{\operatorname{Cox}}
\newcommand{\Bl}{\operatorname{Bl}}

\newcommand{\Supp}{\operatorname{Supp}}
\newcommand{\Char}{\operatorname{char}}

\newenvironment{example}{\begin{varexample}
\begin{normalfont}}{\end{normalfont}
\end{varexample}}
\usepackage[foot]{amsaddr}
\begin{document}
\title{A simplicial approach to effective divisors in $\M_{0,n}$}
\author{Brent Doran$^1$}
\email{brent.doran@math.ethz.ch}
\address{$^1$Departement Mathematik, ETH \\ Z\"urich, Switzerland}

\author{Noah  Giansiracusa$^2$}
\email{noahgian@uga.edu}
\address{$^2$Department of Mathematics, University of Georgia \\
  Athens, GA 30602, USA}

\author{David Jensen$^3$}
\email{dave.jensen@uky.edu}
\address{$^3$Department of Mathematics, University of Kentucky \\
  Lexington, KY 40605, USA}

\maketitle

\vspace{-0.1in}

\begin{abstract}
We study the Cox ring and monoid of effective divisor classes of $\M_{0,n} \cong \Bl\PP^{n-3}$, over a ring $R$.  We provide a bijection between elements of the Cox ring, not divisible by any exceptional divisor section, and pure-dimensional singular simplicial complexes on $\{1,\ldots,n-1\}$ with weights in $R\setminus\{0\}$ satisfying a zero-tension condition.  This leads to a combinatorial criterion, satisfied by many triangulations of closed manifolds, for a divisor class to be among the minimal generators for the effective monoid.  For classes obtained as the strict transform of quadrics, we present a complete classification of minimal generators, generalizing to all $n$ the well-known Keel-Vermeire classes for $n=6$.  We use this classification to construct new divisors with interesting properties for all  $n\ge 7$.
\end{abstract}


\section{Introduction}

Determining the pseudo-effective cone, the closure  $\Eff(X)\subseteq N^1(X)$ of the cone of effective divisor classes on a projective variety $X$, is a familiar and challenging problem in geometry.   A related structure, defined when the Picard group $\Pic(X)$ is finitely generated, is the Cox ring, or total coordinate ring \cite{Cox95,HK00}:
$$ \Cox ( X ) := \bigoplus_{L \in \Pic (X)} \HH^0 (X,L) . $$
This ring is graded by $\Pic(X)$, and the support for this grading is the monoid of effective divisor classes, sometimes denoted $M(X)$, which forms a collection of distinguished lattice points in the effective cone.  Given a set of generators for $\Cox(X)$, their classes in $\Pic(X)$ yield a set of generators for $M(X)$, and the limits of the rays spanned by these classes in $N^1(X)$ generate $\Eff(X)$.  However, $\Cox(X)$ need not be finitely generated, even if $\Eff(X)$ or $M(X)$ is, so computing a presentation of the Cox ring is often not a plausible way to study effective divisor classes.  Moreover, the list of generators one obtains (whether finite or not) from generators of the Cox ring can be massively redundant.  When $\Pic(X)$ is finitely generated and torsion-free, there is a unique minimal generating set for $M(X)$ which is typically quite geometrically meaningful.  For instance, if $X$ is a del Pezzo surface then the minimal generators for $M(X)$ are the classes of the $(-1)$-curves and this provides a bridge to many beautiful areas of representation theory and combinatorics.

The Cox ring, effective cone, and effective monoid are usually studied over the complex numbers, since that is where the strongest geometric tools are applicable \cite{Laz04}.   However, they may be  defined more generally if $X$ is defined over other rings and can encode important arithmetic information in such cases.

In this paper we develop a systematic framework for studying the Cox ring and monoid of effective divisor classes for the moduli space $\M_{0,n}$ of $n$-pointed stable rational curves, valid over any ring $R$.  We crucially rely on the Kapranov-Hassett presentation of $\M_{0,n}$ over $\Spec \mathbb{Z}$ as an iterated blow-up of $\mathbb{P}^{n-3}$ along linear subspaces \cite{Kap93a,Has03}, which in turn leads to a description, over the integers, of $\Cox(\M_{0,n})$ as a subring of invariants of a $\Pic(\M_{0,n})$-graded polynomial ring \cite{DG13}.  As there are no torsion issues, and Picard groups are canonically identified upon base change, there is no choice of isomorphism class of universal torus torsor and the Cox ring is well-defined and may be studied over $R$ via its presentation over $\mathbb{Z}$.   Questions about $\Cox(\M_{0,n})$ date back to the introduction of the notion of a Cox ring \cite{HK00}, and studying the collection of effective divisors in $\M_{0,n}$ has been central to a burgeoning industry for a myriad of moduli spaces \cite{CFM13}.  It has very recently been shown that $\Cox(\M_{0,n})$ is not finitely generated for large $n$ (originally $n \ge 134$  \cite{CT15}, then subsequently extended to $n \ge 13$ \cite{GK14}) when $R$ is an algebraically closed field of characteristic zero; finite generation of $\Eff(\M_{0,n})$ and $M(\M_{0,n})$ remains open for all $n \ge 7$.  Rather than focusing on finite generation, however, we turn to questions such as the following:
\begin{enumerate}
\item How do the Cox ring and effective divisor classes depend on the base ring $R$?
\item Is the first lattice point of every ray of the effective cone an effective divisor class?
\item Does a minimal generating set for the Cox ring yield one for the effective monoid or cone?
\item Are there classes spanning rays outside the previously known subcone of the effective cone?
\item What other geometric structures are related to the minimal generators of these objects?
\end{enumerate}
We find that the answer to \emph{all} of these questions changes drastically when passing from $n \le 6$ to $n \ge 7$.  The following gives a flavor of our findings, all of which are provided by explicit examples.  Question (5) is addressed in \S\ref{sec:Precise}, where detailed versions of these and other results are presented.
\begin{theorem}\label{thm:IntroMain}
Fix $n \ge 7$ and consider $\M_{0,n}$ relative to an arbitrary field $k$.
\begin{enumerate}
\item[(1+2)] There are rays in $N^1(\M_{0,n})$ that are effective in all characteristics, yet their first lattice point is effective if and only if $\Char(k) = 2$;
\item[(3)] The monoid $M(\M_{0,n})$ is not generated by the divisor classes spanning extremal rays of $\Eff(\M_{0,n})$, hence the ring $\Cox(\M_{0,n})$ is not generated by the sections of such divisor classes, when $\Char(k) \ne 2$.
\item[(4)] There are effective rays outside the cone of boundary and hypertree divisors, thereby disproving a conjecture of Castravet and Tevelev \cite{CT13}.
\end{enumerate}
\end{theorem}

\begin{remark}\label{rem:Opie}
While preparing this preprint we were informed that Opie had simultaneously constructed counterexamples to the hypertree conjecture \cite{Opi13}.  We show that, for $n \ge 9$, there are extremal rays outside the cone generated by boundary, hypertree, and Opie's divisors.
\end{remark}

By contrast, for $n \le 6$ the Cox ring has the same set of minimal generators for any field, the classes of these generators form a minimal generating set for the effective monoid, and these classes are the first lattice points of each extremal ray of the effective cone (see \S\ref{sec:Previous}).  For $n \le 5$ these generators are all boundary divisors.  For $n=6$ there are non-boundary generators, and they are all given, in terms of Kapranov's blow-up construction $\M_{0,n} \cong \Bl\PP^{n-3}$, as the strict transform of certain quadric hypersurfaces.  We provide a complete classification of all minimal generators of $M(\M_{0,n})$ corresponding to quadric hypersurfaces, for any $n$; this provides a rich enough supply of divisor classes to construct all the examples described in Theorem \ref{thm:IntroMain}.

The basis of our framework is a result in \cite{DG13} that $\Cox(\M_{0,n})$ is isomorphic, over $\Spec\mathbb{Z}$, to the ring of $\Ga$-invariants for an action on a polynomial ring. This polynomial ring is the Cox ring of a toric variety, call it $X_n$, with the same Picard group as $\M_{0,n}$.  The $\Ga$-action on $\Cox(X_n)$ is induced by one on the universal torus torsor over $X_n$, an open subset of affine space, rather than an action on $X_n$ itself.  Since the grading by this common Picard group on the two Cox rings is compatible, a divisor class on $\M_{0,n}$ is effective if and only if the corresponding class on $X_n$ admits a $\Ga$-invariant section.  We encode homogeneous polynomials in $\Cox(X_n)$ by a singular simplicial complex whose simplices are in bijection with the terms of the polynomial and are decorated with the corresponding coefficients.  Invariance under the $\Ga$-action then translates into a zero-tension condition that for each face of a simplex in the complex, the sum of $R$-weights of simplices containing that face is zero.  As we shall see, this perspective of ``balanced complexes'' yields a powerful new perspective on the Cox ring $\Cox(\M_{0,n})$ and effective monoid $M(\M_{0,n})$, and, like the case of $(-1)$-curves on a del Pezzo surface, it ties the geometry of effective divisors in $\M_{0,n}$ to many beautiful classical constructions, such as triangulations of the sphere and other closed manifolds.  We list several open questions and areas of further study arising from this perspective in \S\ref{sec:Future}.

\subsection{Previously known results and constructions}\label{sec:Previous}

Throughout this paper we fix an isomorphism $\M_{0,n} \cong \Bl\PP^{n-3}$, i.e., a choice of $\psi$-class.  This isomorphism was originally provided by Kapranov over $\Spec\mathbb{C}$ \cite{Kap93a}, but work of Hassett extends it to be defined over $\Spec\mathbb{Z}$ \cite[\S6]{Has03}.  This yields a $\mathbb{Z}$-grading on $\Cox(\M_{0,n})$ and $\Pic(\M_{0,n})$; indeed, by the \emph{degree} of a divisor or its class we shall mean the image under the map $\Pic(\M_{0,n}) \to \Pic(\PP^{n-3}) \cong \mathbb{Z}$ induced by the above blow-up presentation.  Thus, the degree of an effective divisor is zero if it is a union of exceptional divisors and it is simply the degree of its image as a hypersurface in $\PP^{n-3}$ otherwise.

The first conjecture concerning effective divisors in $\M_{0,n}$ was given by Fulton, although formally stated by Keel and McKernan in \cite{KM96}.  It asserts that $\Eff(\M_{0,n})$ is spanned by the classes of boundary divisors, extending the del Pezzo situation of $n=5$ and inspired by the geometry of toric varieties where boundary strata do indeed span all effective classes.  Put another way, this conjecture says that the effective cone is generated by classes of degree at most one.  However, Keel and Vermeire simultaneously found degree two counterexamples for $n=6$ \cite{Ver02,GKM02}. Hassett and Tschinkel, with the help of computer calculation, verified that for $n=6$ the Keel-Vermeire classes, together with the boundary classes, do in fact span the entire effective cone \cite{HT02}.

The question remained of whether the unique (up to scalar) sections of these divisor classes generate $\Cox(\M_{0,6})$.  Here Castravet performed a comprehensive analysis of this particular case and found a positive answer to this question \cite{Cas09}.  This implies that the effective monoid $M(\M_{0,6})$ is also minimally generated by the boundary and Keel-Vermeire classes, so essentially everything is known here and attention now turns to $n > 6$.  Many experts believed the Keel-Vermeire classes in $n=6$ must be part of a more general construction valid for all $n$, yet this vast generalization proved rather elusive.  A breakthrough was made in 2009 when Castravet and Tevelev found a construction (summarized in \S\ref{sec:CTbackground}) that yields all Keel-Vermeire classes in $n=6$ and a huge number of new extremal rays for $n \ge 7$.  They conjectured that the boundary classes together with these so-called hypertree divisor classes span $\Eff(\M_{0,n})$ for all $n$ \cite{CT13}.  Additionally, they raised the question (and called it a ``pipe dream'') of whether, as with $n=6$, it is always the case that the Cox ring is generated by sections of divisor classes spanning extremal rays of the effective cone.  One could weaken this dream slightly and ask for the effective monoid $M(\M_{0,n})$ to be generated by classes lying along extremal rays.  As we shall see below, these dreams fail for all $n \ge 7$, as does the hypertree conjecture (see also Remark \ref{rem:Opie} concerning Opie's exciting paper \cite{Opi13}).

Along with the remarkable recent result of Castravet and Tevelev that $\Cox(\M_{0,n})$ is not finitely generated for $n \ge 134$ \cite{CT15}, or in other words, that Hu and Keel's dream for $\M_{0,n}$ is also false \cite{HK00}, we are left once again in a quandary: one hopes that there is some order to the seemingly chaotic and unwieldy collection of effective divisors and divisor classes on $\M_{0,n}$, yet the challenge is to find a structure that best describes it and reveals how truly complex it is.

\subsection{Precise statement of new results and constructions}\label{sec:Precise}

For us, a (possibly singular) $d$-simplex is a multiset of cardinality $d+1$, and a pure-dimensional simplicial complex of dimension $d$, or ``$d$-complex'', with vertex set $[n-1] := \{1,\ldots,n-1\}$ is any set of $d$-simplices with entries in $[n-1]$.  Thus, a 1-complex is a graph with loops allowed but multiple edges disallowed.   For a multiset $S$, the number of times $i\in S$ occurs is its multiplicity, denoted $\mult_i(S)$.  We call a complex ``non-singular'' if every element of every simplex in it has multiplicity at most one.

Fix a ring $R$ and consider $\M_{0,n}$ relative to $\Spec R$.  A weighted complex is an assignment of a nonzero element of $R$ to each simplex in a complex, and a weighted complex is ``balanced'' if, for each face, the sum of the weights of simplices containing it is zero (see Definition \ref{def:balanced} for details).  It is ``balanceable'' if there exists a collection of elements of $R\setminus\{0\}$ satisfying this property.  For any $d$-complex $\Delta$ on $[n-1]$ we define a degree $d+1$ divisor class $D_\Delta$ as follows: \[D_\Delta := (d+1)H - \sum_I \left(d+1-\max_{\sigma\in\Delta}\left\{\sum_{i\in I}\mult_i(\sigma)\right\}\right)E_I \in \Pic(\M_{0,n}).\]
Here $I\subseteq [n-1]$ satisfies $1 \le |I| \le n-4$ and $E_I$ denotes the corresponding exceptional divisor.  The following result is where we crucially use the identification in \cite{DG13} (and summarized in \S\ref{sec:CoxInvar}) of $\Cox(\M_{0,n})$ with an explicit ring of $\Ga$-invariants:

\begin{theorem}\label{thm:IntroCorresp}
For any $d \ge 0$ and $n \ge 5$, there is a natural bijection between degree $d+1$ multi-homogeneous elements of $\Cox(\M_{0,n})$, not divisible by any exceptional divisor section, and balanced $d$-complexes on $[n-1]$.  All balancings on a complex $\Delta$ correspond to elements with class $D_\Delta$.
\end{theorem}

We call a complex ``minimal'' if it is balanceable yet no proper subcomplex is balanceable.  For a $d_1$-complex $\Delta_1$ and a $d_2$-complex $\Delta_2$, their product is the $(d_1+d_2+1)$-complex $\Delta_1\cdot \Delta_2$ with simplices $\sigma_1\cup \sigma_2$ where $\sigma_i\in\Delta_i$.  Recall that the monoid $M(\M_{0,n})$ of effective divisor classes has a unique collection of minimal generators (cf., Remark \ref{rem:MonIrr}).  We can now state our main results on complexes and divisor classes:

\begin{theorem}
Fix integers $d,n$ as above.
\begin{enumerate}
\item Let $D \in \Pic(\M_{0,n})$ be a class such that $D-E_I$ is not effective for any $I$.  Then $D$ is effective if and only if there is a balanceable complex $\Delta$ with $D_\Delta = D$.
\end{enumerate}
Assume now that $d \le n-5$ and $\Delta$ is a non-singular $d$-complex.
\begin{enumerate}
\item[(2)] Suppose that every proper subcomplex $\Delta'\subsetneq \Delta$ is not balanceable.  Then $D_\Delta$ is effective if and only if $\Delta$ is balanceable.
\item[(3)] If $\Delta$ is minimal and $R$ is a field then $h^0(\M_{0,n},D_\Delta) = 1$.  If, moreover, there is no decomposition $\Delta = \Delta_1 \cdot \Delta_2$ with $\Delta_i$ minimal non-singular complexes supported on disjoint subsets of $[n-1]$, then $D_\Delta$ is a minimal generator of $M(\M_{0,n})$ and every generating set for $\Cox(\M_{0,n})$ includes the unique (up to scalar) section of $D_\Delta$.
\end{enumerate}
\end{theorem}

These results, as well as Theorem \ref{thm:IntroCorresp}, are all stated and proven in \S\ref{sec:Complexes}.  Turning now to 1-complexes (i.e., graphs) and their corresponding degree two divisor classes (i.e., strict transforms of quadric hypersurfaces) we have the following (see \S\ref{sec:Deg2}):

\begin{theorem}
Fix a base field $k$.  The minimal generators of $M(\M_{0,n})$ with degree two are precisely the classes $D_\Delta$ where $\Delta$ is either (cf., Figure \ref{fig:minimal}):
\begin{enumerate}
\item  an $m$-gon with $m \ge 6$ even,
\item  an $m_1$-gon meeting an $m_2$-gon at a single vertex, with both $m_i \ge 3$ odd, or
\item[(3a)]  (when $\Char(k)=2$) the disjoint union of an $m_1$-gon and an $m_2$-gon, with $m_i \ge 3$ odd;
\item[(3b)]  (when $\Char(k) \ne 2$) an $m_1$-gon connected by a path of any length of degree two vertices to an $m_2$-gon, with $m_i \ge 3$ odd.
\end{enumerate}
\end{theorem}

In particular, for any $n \ge 6$ we have non-boundary classes, not pulled-back along a forgetful map, that are minimal generators over any field.  The divisor classes appearing in this classification, and closely related ones, exhibit some interesting previously unobserved phenomena---specifically, that of Theorem \ref{thm:IntroMain}.  If $\Delta$ is a 1-complex given by two triangles connected by a single edge (Figure \ref{fig:tribritri}) then $2D_\Delta$ is a sum of boundary classes (cf., Example \ref{ex:2Dboundary}) so the class $D_\Delta$ does not span an extremal ray of $\Eff(\M_{0,7})$, even though it is a minimal generator of $M(\M_{0,7})$ and its unique (up to scalar) section is a necessary generator for $\Cox(\M_{0,7})$, when $\Char(k) \ne 2$.  If $\Delta$ is the 1-complex of edges in a disjoint union of two triangles, then $D_\Delta$ is effective if and only $\Char(k)=2$, yet it turns out (see \S\ref{subsec:twotri}) that $2D_\Delta$ is effective in all characteristics.  Moreover, the ray spanned by this class lies outside the cone of boundary and hypertree divisors (Corollary \ref{cor:NEqualsSeven}), so it yields an $n=7$ counterexample to the Castravet-Tevelev conjecture and its pull-backs along the forgetful morphisms thus yield counterexamples for all $n$.

If $\Delta$ is the 1-complex of edges of an octagon, then $D_\Delta$ spans an extremal ray of $\Eff(\M_{0,9})$ which is also a counterexample to the hypertree conjecture (\S\ref{subsec:oct}).  We prove extremality by constructing a curve that covers this divisor and intersects it negatively.  To see that this ray is distinct from all the extremal rays spanned by hypertree divisors, we prove in \S\ref{sec:Deg} that the only hypertree divisors of degree two for some choice of $\psi$-class are the unique $n=6$ hypertree (the Keel-Vermeire class, which we identify as $D_{\Delta'}$ where $\Delta'$ is the edges of a pair of triangles meeting at a single vertex) and the unique $n=7$ hypertree (which we identify as $D_{\Delta'}$ where $\Delta'$ is the edges of a hexagon).  Note that this extremal ray spanned by $D_\Delta$ for $\Delta$ the octagon is also distinct from the extremal rays recently found by Opie, since those all have degree greater than two \cite{Opi13}.  In \S\ref{sec:Octagon} we illustrate the correspondence of Theorem \ref{thm:IntroCorresp} with this particular choice of $\Delta$.

In \S\ref{sec:Deg3} we give some examples of minimal complexes yielding divisor classes of degree three that are minimal generators for the effective monoid.  One of these, the faces of a cycle of tetrahedra linked together in a precise way (see Figure \ref{fig:TetCyc}), turns out to correspond to the unique $n=7$ hypertree divisor (for a different choice of $\psi$-class than the one used above) when there are three tetrahedra in the cycle.  The other examples we present are well-known triangulations of the sphere and torus.  It would be interesting to understand more systematically the link between such divisors and triangulations of closed manifolds.  We conclude in \S\ref{sec:Future} with a list of open questions and areas for further research arising from the investigations in this paper.



\subsection{A quadric example: the octagon}\label{sec:Octagon}

We present here an example of the correspondence in Theorem \ref{thm:IntroCorresp} in the case of the 1-complex $\Delta$ given by the edges of an octagon.  The associated divisor class $D_\Delta \in \Pic(\M_{0,9})$ is our degree two counterexample to the hypertree conjecture for $\M_{0,9} = \Bl\PP^6$.  From our $\Ga$-invariant perspective, one should view the coordinates on $\PP^6$ as differences $y_i-y_j$ of the homogeneous coordinates $y_1,\ldots,y_8$ on $\PP^7$ (see \S\ref{sec:CoxInvar}).  If we set \[\Delta = \{\{1,2\},\{2,3\},\ldots,\{7,8\},\{8,1\}\}\] then this can be balanced by assigning these edges alternating elements of $\{\pm 1\} \subseteq R$ (in fact, every balanced $R$-weighting is a constant multiple of this one).  Associated to this weighted complex is the quadric hypersurface in $\PP^7$ defined by \[y_1y_2 - y_2y_3 + \ldots + y_7y_8 - y_8y_1 = 0,\]  which is a cone over a quadric hypersurface in $\PP^6$.  If we label the 7 coordinate points as $q_1,\ldots,q_7$ and the ``general'' point $q_8 := [1 : \cdots : 1]$, then in the Kapranov blow-up the strict transform of our quadric has class \[2H - \sum_{\substack{I \subset \{1,\ldots,8\} \\ ~1 \le |I| \le 5 \\ J \nsubseteq I~\forall J \in \Delta}} E_I \in \Pic(\M_{0,9}).\]  Geometrically, a codimension $\ge 2$ linear subspace spanned by a collection of the $q_i$ is contained in this quadric unless two of the points are connected by an edge in the octagon.

\begin{ack}
We thank Aravind Asok, Ana-Maria Castravet, Charles Doran, Emily Frey, Jeff Giansiracusa, Angela Gibney, Sean Keel, Frances Kirwan, Rahul Pandharipande, Bernd Sturmfels, and Jenia Tevelev for useful discussions, and Andrey Novoseltsev for numerous computer calculations.  We thank Abramovich, Castravet, and Tevelev in particular for comments on an earlier version of this paper, and the organizers of the COMB IV conference, August 2013, for inviting us to present our hypertree counterexamples.  The first author was supported in part by the SNF, the second author by an NSF postdoctoral fellowship, and the third author by a Simons postdoctoral fellowship.  We also thank the ESI in Vienna for hosting all three of us for a portion of these investigations and FIM in Z\"urich for hosting the second author.
\end{ack}


\section{Background: the Cox ring as a ring of invariants}\label{sec:CoxInvar}

Here we briefly summarize the construction and result of \cite{DG13} that we rely on throughout.  The reader is referred to that paper for all details.  Let $X_n$ be the toric variety obtained as the iterated blow-up of $\PP^{n-2}$ along all codimension $\ge 3$ coordinate linear subspaces, in order of increasing dimension.  This can be viewed as the result of performing Kapranov's iterated blow-up construction of $\M_{0,n}$ but in $\PP^{n-2}$ instead of $\PP^{n-3}$, so that there is a projectivity sending all of the blown-up points to coordinate points, not just all but one of them.  The Picard groups of $\M_{0,n}$ and $X_n$ can be identified once a $\psi$-class on $\M_{0,n}$ is chosen to induce an isomorphism $\M_{0,n} \cong \Bl\PP^{n-3}$.  Indeed, \[\Pic(\M_{0,n}) =  \mathbb{Z}H\bigoplus_{1 \le |I| \le n-4}\mathbb{Z} E_I = \Pic(X_n),\] where $H$ is the pull-back of a general hyperplane and the $E_I$, $I \subset [n-1]$, are the exceptional divisors.

The Cox ring of $X_n$, as with any toric variety, is a polynomial ring \cite{Cox95}; write it as \[\Cox(X_n) = k[y_1,\ldots,y_{n-1},(x_I)_{1 \le |I| \le n-4}],\] where the $y_i$ define the strict transforms of the coordinate hyperplanes in $\PP^{n-2}$ and $x_I$ is a section of $E_I$.  The $\Pic(X_n)$-grading is given by \begin{equation}\label{eq:PicDeg}[x_I] = E_I\text{ and } [y_i] = H - \sum_{I \not\ni i}E_I.\end{equation} The non-reductive group $\Ga$ admits an action (non-linearizable for $n \ge 6$) \[x_I \mapsto x_I \text{ and } y_i \mapsto y_i + s\prod_{I \ni i}x_I,~s\in \Ga\] satisfying $\Cox(X_n)^{\Ga} = \Cox(\M_{0,n})$ \cite[Corollary 1.3(1)]{DG13}.  The orbits for the induced $\Ga$-action on affine space yield, upon quotienting by the Neron-Severi torus $\Hom(\Pic(X_n),\mathbb{G}_m)$, the fibers of a rational map $X_n \dashrightarrow \M_{0,n}$.  The restriction of this map to each exceptional divisor and to the complement of the exceptional divisors is a linear projection.  For instance, when $n=5$ it is the map $\Bl_{4\text{pts}}\PP^3 \dashrightarrow \Bl_{4\text{pts}} \PP^2$ given by projecting linearly from a point $q$ in the base $\PP^3$ and from the point in each exceptional $\PP^2$ that is the intersection with the strict transform of the line from $q$.

Inverting all the $x_I$ roughly corresponds to removing the exceptional divisors, in which case the geometry is simply a linear projection $\PP^{n-2} \dashrightarrow \PP^{n-3}$.  On this localization of $\Cox(X_n)$ we have the uniform translation action \[\frac{y_i}{z_i} \mapsto \frac{y_i}{z_i} + s, \text{ where } z_i := \prod_{I \ni i}x_I,\] so the homogeneous coordinates on $\PP^{n-3}$ are the differences $\frac{y_i}{z_i} - \frac{y_j}{z_j}$ (these generate the ring of invariants for this action).  Any hypersurface in $\PP^{n-3}$ is given by a polynomial in these differences, and its strict transform in $\Bl\PP^{n-3}$ is given by clearing the denominators.  More formally, we have:
\begin{theorem}[{DG13}, Corollary 1.3]\label{Thm:DGcor}
For all $n \ge 5$,
\[\Cox(\M_{0,n}) = k[(x_I^{\pm 1})_{1 \le |I| \le n-4},(\frac{y_i}{z_i} - \frac{y_j}{z_j})_{i,j \le n-1}] \cap \Cox(X_n).\]
\end{theorem}
By Equation (\ref{eq:PicDeg}) we have $[\frac{y_i}{z_i} - \frac{y_j}{z_j}] = H - \sum_{1 \le |I| \le n-4} E_I$.  The boundary divisors are given by the $x_I$ and the binomials obtained by clearing denominators of the differences $\frac{y_i}{z_i} - \frac{y_j}{z_j}$.  Elements of the Cox ring with class that is not a sum of boundary classes are given by clearing denominators of a polynomial in the differences $\frac{y_i}{z_i} - \frac{y_j}{z_j}$ where some cancellation occurs when expanding it out.  For instance, the Keel-Vermeire divisors are obtained from \begin{equation}\label{eq:KV}(\frac{y_1}{z_1} - \frac{y_2}{z_2})(\frac{y_3}{z_3} - \frac{y_4}{z_4}) - (\frac{y_1}{z_1} - \frac{y_3}{z_3})(\frac{y_2}{z_2} - \frac{y_5}{z_5})\end{equation} and its permutations.


\section{Weighted simplicial complexes}\label{sec:Complexes}

Given a multiset $S$, we denote its cardinality (counting elements according to their multiplicity) by $|S|$.  The \emph{support}, denoted $\Supp(S)$, is the set of elements appearing in $S$, irrespective of multiplicity.  Thus $|\Supp(S)| \le |S|$, with equality if and only if $S$ has no repeated elements.

\begin{definition}
For a positive integer $d$ and a set $A$, by a \emph{$d$-simplex on $A$} we shall mean a multiset of cardinality $d+1$ with entries in $A$.  It is \emph{singular} if $|\Supp(\sigma)| < |\sigma|$, and \emph{non-singular} otherwise.  A \emph{simplicial complex of pure-dimension $d$ on $A$}, or ``$d$-complex'' for short, is any set of $d$-simplices on $A$.  We call a complex \emph{singular} if any of the simplices comprising it is singular, otherwise the complex is \emph{non-singular}.\end{definition}

For instance, a 1-complex is a graph in which loops are allowed (and remembered as edges of the form $\{v,v\}$), but multi-edges are disallowed.

\begin{definition}\label{def:balanced}
For a ring $R$, we say that a $d$-complex $\Delta = \{\sigma_1,\ldots,\sigma_r\}$ on $A$ is \emph{$R$-weighted} if each $d$-simplex $\sigma_i$ is assigned a non-zero element $w_i \in R$.  The weighted complex $(\Delta,\{w_i\})$ is \emph{balanced in degree j} if for each multiset $S$ with $|S|=j$ and $\Supp(S)\subseteq A$, we have \[\sum_{\sigma_i \supseteq S}w_i = 0.\]  A weighted $d$-complex is \emph{balanced} if it is balanced for each $j\in \{0,\ldots,d-1\}$, and an arbitrary $d$-complex is \emph{balanceable over $R$} if there exists nonzero elements of $R$ making it balanced as an $R$-weighted $d$-complex.
\end{definition}

\begin{remark} The sum in the above balancing condition incorporates the multiset structure of singular simplices, so for example a loop $\sigma = \{i,i\}$ in a 1-complex with weight $w$ contributes weight $2w$ for $S=\{i\}$ because we view $\{i\}$ as a subset of the multiset $\{i,i\}$ with multiplicity two.
\end{remark}

Being balanced in degree 0 simply means the total weight is zero: $\sum_{i=1}^r w_i = 0$.  This condition is implied by the other balancing conditions in many situations, as the following result illustrates.

\begin{proposition}\label{prop:chardep}
Fix integers $1\le \ell \le d$.  If the image of $\binom{d+1}{\ell} \in \mathbb{Z}$ under the map $\mathbb{Z} \rightarrow R$ is a non-zero-divisor, then an $R$-weighted $d$-complex is balanced in degree 0 if it is balanced in degree $\ell$.
\end{proposition}

\begin{proof}
Consider a weighted $d$-complex with notation $(\Delta,\{w_i\})$ as above.  Then
\[\binom{d+1}{\ell}\sum_{i = 1}^r w_i = \sum_{|S| = \ell}\sum_{\sigma_i \supseteq S} w_i \] since each $d$-simplex has $\binom{d+1}{\ell}$ faces spanned by $\ell$ vertices.  The claim now follows immediately.
\end{proof}

Recall from the introduction that we fix an isomorphism $\M_{0,n} \cong \Bl\PP^{n-3}$, for $n \ge 5$, and hence a notion of degree of divisors and their classes.  We also use the notation from \S\ref{sec:CoxInvar} and work over an arbitrary ring $R$---that is, all weightings of complexes are $R$-weightings, unless otherwise stated, and the variety $\M_{0,n}$ is considered relative to $\Spec R$; in particular, $\Cox(\M_{0,n})$ is an $R$-algebra.

\begin{theorem}\label{thm:CoxCom}
For $d \ge 1$, there is a bijection between degree $d$ homogeneous elements of $\Cox(\M_{0,n})$, not divisible by any exceptional divisor section $x_I$, and balanced $(d-1)$-complexes on $[n-1]$.
\end{theorem}

\begin{proof}
Such an element of the Cox ring is a section of a divisor class corresponding to the strict transform in $\Bl\PP^{n-3} \cong \M_{0,n}$ of a degree $d$ hypersurface in $\PP^{n-3}$.  As discussed in \S\ref{sec:CoxInvar}, these hypersurfaces in $\PP^{n-3}$ can be viewed as degree $d$ polynomials in the rational functions $\frac{y_i}{z_i}$ that are invariant under the uniform translation action $\frac{y_i}{z_i} \mapsto \frac{y_i}{z_i} + s$, $s\in \Ga$, and their strict transforms are obtained by clearing denominators.  Therefore, it suffices to show that balanced complexes are in bijection with invariant homogeneous polynomials in these coordinates.

For each balanced $(d-1)$-complex $\underline{\Delta} = (\Delta,\{w_i\}_{i=1}^r)$, consider the Laurent polynomial
\[f_{\underline{\Delta}} := \sum_{i=1}^r w_i\frac{y_{\sigma_i}}{z_{\sigma_i}},\] where $y_{\sigma_i} = \prod_{j\in \sigma_i} y_j$ (resp. $z_{\sigma_i}$) is the degree $d$ monomial given by the obvious multi-index notation.

The $\Ga$-action sends $f_{\underline{\Delta}}$ to an expression that is a degree $d$ polynomial in the parameter $s\in\Ga$.  The coefficient of the $s^j$-term is a homogeneous polynomial in the $\frac{y_i}{z_i}$ of degree $d-j$.  Given a multiset $S$ of cardinality $d-j$, the sum \[\sum_{\sigma_i \supseteq S} w_i\] appearing in the degree $d-j$ balancing condition is precisely the coefficient of the term $\frac{y_S}{z_S}$ appearing in this homogeneous polynomial.  Therefore, the $s^j$-term of our polynomial vanishes if and only if $\underline{\Delta}$ is balanced in degree $d-j$, and hence $f_{\underline{\Delta}}$ is $\Ga$-invariant if and only if $\underline{\Delta}$ is balanced.
\end{proof}

\begin{remark}\label{rem:LaurentClass}
The Laurent polynomial $f_{\underline{\Delta}}$ appearing in the above proof has divisor class independent of the complex $\Delta$ and the weights $w_i$; indeed, by recalling the $\Pic(\M_{0,n})$-grading on $\Cox(X_n)$ discussed in \S\ref{sec:CoxInvar} one sees immediately that $[y_\sigma] - [z_\sigma] = dH - d\sum E_I$ for any $(d-1)$-simplex $\sigma$, hence $f_{\underline{\Delta}}$ has this class as well.
\end{remark}

\begin{definition}
The \emph{support} of a complex $\Delta$ on $A$ is the collection of indices appearing in the simplices of $\Delta$: \[\Supp(\Delta) := \bigcup_{\sigma\in \Delta}\Supp(\sigma) \subseteq A.\]
\end{definition}

Clearly a complex $\Delta$ on $A$ may be viewed as a complex on $\Supp(\Delta) \subseteq A$.

\begin{lemma}\label{lem:ComPull}
If $\underline{\Delta}$ is a balanced complex on $[n-1]$, corresponding to $f\in\Cox(\M_{0,n})$, then its restriction to $\Supp(\Delta)$ is balanced; if $g\in\Cox(\M_{0,\Supp(\Delta)\sqcup \{*\}})$ is the corresponding polynomial, then $f = \pi^*g$, where $\pi : \M_{0,[n-1]\sqcup\{*\}} \rightarrow \M_{0,\Supp(\Delta)\sqcup\{*\}}$ is the induced forgetful map.
\end{lemma}

\begin{proof}
The balancing statement is trivial, so we focus on the pull-back statement.  By \cite[Proposition 2.7]{Kap93b}, this forgetful map corresponds to linear projection \[\pi^\circ : \mathbb{P}^{n-3} \dashrightarrow \mathbb{P}^{|\Supp(\Delta)|-2}\] from the subspace spanned by the subset of the $n-1$ points $e_1,\ldots, e_{n-2}, \sum e_i \in \mathbb{P}^{n-3}$ corresponding to the ``forgotten'' points.  By relabelling if necessary, we can assume without loss of generality that the projection is only from coordinate points $e_i$ and not the point $\sum e_i$.  Let us write \[z'_i = \prod_{i \in I \subset \Supp\Delta}x_I \in \Cox(\M_{0,\Supp(\Delta)\sqcup\{*\}}),\] so that our coordinates on $\mathbb{P}^{|\Supp(\Delta)|-2}$ are the differences $\frac{y_i}{z'_i} - \frac{y_j}{z'_j}$ for $i,j\in \Supp(\Delta)$.  Since our coordinates on $\PP^{n-3}$ are the differences $\frac{y_i}{z_i} - \frac{y_j}{z_j}$ for $i,j\in [n-1]$, then the fact that we are only considering a coordinate linear projection immediately implies that \[(\pi^\circ)^* f_{\underline{\Delta}|_{\Supp(\Delta)}} = f_{\underline{\Delta}}.\]  Since taking strict transforms of divisors commutes with pullback here, the result follows.
\end{proof}

\begin{remark}\label{rem:sum}
Given balanced $(d-1)$-complexes $\underline{\Delta}$ and $\underline{\Delta}'$, it is easy to see that by summing their $R$-weightings one obtains a balanced $(d-1)$-complex supported on a subset of $\Delta\cup\Delta'$.  However, the correspondence in Theorem \ref{thm:CoxCom} is not generally an additive homomorphism with respect to this structure.  Indeed, if there is a simplex $\sigma\in \Delta\cap\Delta'$ and the associated weights satisfy $w_\sigma + w_{\sigma}' = 0$, then it may not be the case that clearing denominators of $f_{\underline{\Delta}},f_{\underline{\Delta}'}$ commutes with taking their sum.
\end{remark}


\subsection{Effective divisor classes}

The following lemma says that the class in $\Pic(\M_{0,n})$ of the polynomial associated to a balanced complex is independent of the balancing weights.  Thus, associated to each balanceable complex $\Delta$ is a well-defined divisor class that we denote by $D_\Delta\in \Pic(\M_{0,n})$.  First some notation: given a multiset $S$ and an element $i\in \Supp(S)$, we define the \emph{multiplicity} of $i$ in $S$, denoted $\mult_i(S)$, to be the number of times $i$ occurs in $S$.

\begin{lemma}\label{lem:DelClass}
All balanced weightings on a $(d-1)$-complex $\Delta$ yield homogeneous elements in $\Cox(\M_{0,n})$ that are sections of the same line bundle isomorphism class.  Specifically, \[D_\Delta = dH - \sum_I \left(d-\max_{\sigma\in\Delta}\left\{\sum_{i\in I}\mult_i(\sigma)\right\}\right)E_I \in \Pic(\M_{0,n}).\]
\end{lemma}

\begin{proof}
This follows directly from the proof of Theorem \ref{thm:CoxCom} and Remark \ref{rem:LaurentClass}.  Indeed, prior to clearing denominators any Laurent polynomial $f_{\underline{\Delta}}$ has class $dH - \sum dE_I$, and the above formula encodes the effect on this class of minimally clearing denominators.
\end{proof}

\begin{definition}
The divisor class $D_\Delta\in \Pic(\M_{0,n})$ associated to any complex $\Delta$ on $[n-1]$, not just a balanceable complex, is given by the formula in Lemma \ref{lem:DelClass}.
\end{definition}

Some caution is needed here, as it is \emph{not} true that $D_\Delta$ is effective if and only if $\Delta$ is balanceable.  Moreover, we cannot associate to each effective divisor class a well-defined balanceable complex, since as the following example illustrates, different balanceable complexes may yield the same divisor class.  In other words, by varying the weights on a balanced complex one obtains a collection of linearly equivalent effective divisors that need not form a complete linear system.

\begin{example}
The 1-complex on $\{1,2,3,4\}$ given by a square $(\{1,2\},\{2,3\},\{3,4\},\{1,4\})$ with weights $(1,-1,1,-1)$ is clearly balanced, as is the complex $((\{1,2\},\{2,3\},\{3,4\},\{1,4\},\{1,3\},\{2,4\})$ obtained by adding the diagonals and weighting by $(1,1,1,1,-2,-2)$ (see Figure \ref{fig:twosquare}).  The corresponding pair of effective divisors both have class $2H - \sum_{i=1}^4 E_i \in \Pic(\M_{0,5})$.
\end{example}

\begin{figure}\begin{center}
\scalebox{1.0}{\input{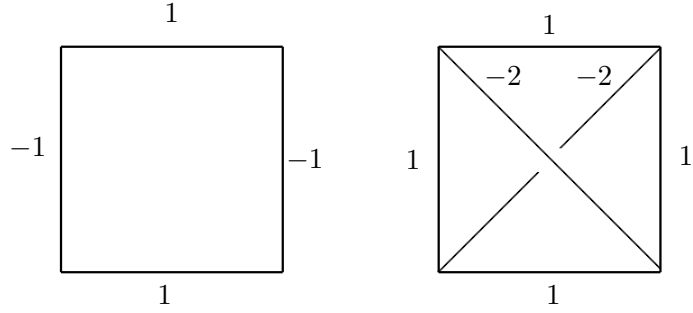}}
\caption{Two balanced 1-complexes with the same class in $\Pic(\M_{0,5}).$}
\label{fig:twosquare}
\end{center}\end{figure}

There are, however, conditions that guarantee there is unique complex of a given divisor class.

\begin{proposition}\label{prop:UniCom}
Let $\Delta$ be a non-singular $(d-1)$-complex, let $\Delta'$ be any complex, and suppose moreover that $d \le n-4$.  Then $D_\Delta = D_{\Delta'}$ if and only if $\Delta' = \Delta$.
\end{proposition}

\begin{proof}
Clearly, we may assume that $\Delta'$ is a $(d-1)$-complex.  Moreover, by Lemma \ref{lem:nonsing} below we can assume $\Delta'$ is non-singular.  Let $\sigma \in (\Delta\setminus\Delta')\cup(\Delta'\setminus\Delta)$.  Since $\sigma$ is non-singular it can be regarded as a set, not just a multiset, and the hypothesis $d \le n-4$ implies that there exists a variable $x_{\sigma}\in\Cox(X_n)$.  Then by the formula defining $D_\Delta$ we see that the corresponding class $E_\sigma$ has coefficient zero in one of $D_\Delta,D_{\Delta'}$ and negative coefficient in the other.
\end{proof}

\begin{lemma}\label{lem:nonsing}
If $\Delta$ is a non-singular $(d-1)$-complex, and $\Delta'$ is a singular $(d-1)$-complex, then $D_\Delta$ cannot be written as a sum of $D_{\Delta'}$ and any collection (even the empty one) of exceptional classes $E_I$.
\end{lemma}

\begin{proof}
By assumption, there is a simplex $\sigma \in \Delta'$ with an index $i$ of multiplicity at least two, whereas $i$ has multiplicity at most one in all simplices of $\Delta$.  Therefore, the coefficient of $E_i$ is strictly greater in $D_{\Delta'}$ than in $D_\Delta$.
\end{proof}

\begin{proposition}\label{prop:EffIff}
Let $\Delta$ be a non-singular $(d-1)$-complex, with $d \le n-4$, and suppose every proper subcomplex $\Delta'\subsetneq \Delta$ is not balanceable.  Then $D_\Delta$ is effective if and only if $\Delta$ is balanceable.
\end{proposition}

\begin{proof}
If $\Delta$ is balanceable then by definition $D_\Delta$ is effective.  Conversely, suppose $\Delta$ is not balanceable.  By Theorem \ref{thm:CoxCom}, any section of $D_\Delta$ then corresponds to either (1) a homogeneous element $f\in\Cox(\M_{0,n})$ divisible by some $x_I$, or (2) a balanced weighting on a complex distinct from $\Delta$.  The second case is ruled out by Proposition \ref{prop:UniCom}, so consider the first case.  By factoring out as many $x_I$ as possible we obtain $f = (\prod x_I)g$ where $g\in\Cox(\M_{0,n})$ corresponds to a balanced complex with underlying complex $\Delta'$.  Then we have \[D_{\Delta} = [f] = D_{\Delta'} + \sum E_I,\] so $\Delta' \ne \Delta$ and by Lemma \ref{lem:nonsing} we know that $\Delta'$ is non-singular.  We claim that $\Delta' \subsetneq \Delta$.  If $\sigma\in \Delta'$ then, since $\sigma$ is non-singular and $d \le n-4$, there is a variable $x_\sigma \in \Cox(X_n)$.  We see from the formula in Lemma \ref{lem:DelClass} that the coefficient of $E_\sigma$ in $D_{\Delta'}$ is zero, hence it is also zero in $D_\Delta$, so by the same formula we must have $\sigma\in\Delta$.
\end{proof}

\begin{remark}
Recall that the set of effective divisor classes for a smooth, projective variety $X$ naturally forms a monoid, denoted $M(X)$.  This monoid is sharp, meaning that the only unit is zero, since a nontrivial line bundle and its dual cannot both admit a global section.  If $\Pic(X)$ is finitely generated and torsion-free, then $M(X)$ is also integral and all non-empty subsets contain a minimal element, so by \cite[Remark 2.1.4]{Ogu06} there is a unique minimal generating set, namely, the irreducible elements of the monoid.  Recall that an element of a sharp additive monoid is \emph{irreducible} if it is nonzero and it cannot be written as a sum of elements unless one of them is zero.
\end{remark}

In the context of $\M_{0,n} \cong \Bl\PP^{n-3}$, an obvious condition for a class to be reducible is if one can subtract off any exceptional divisors from it.  For most purposes it therefore suffices to consider effective classes for which no exceptional classes can be removed; the following result says these classes are completely determined by our framework of balanced complexes.

\begin{corollary}\label{cor:AllClass}
Let $D\in \Pic(\M_{0,n})$ be a class such that $D-E_I$ is not effective for any $I$.  Then $D$ is effective if and only if there is a balanceable complex $\Delta$ with $D_\Delta = D$.
\end{corollary}

\begin{proof}
If there is a balanceable complex $\Delta$ with $D_\Delta = D$, then obviously $D$ is effective.  Conversely, suppose $D$ is effective, and consider a section $f \in |D|$.  If $f$ is divisible by some $x_I$, then $D-E_I$ is effective, contradicting the hypothesis.  Thus $f$ satisfies the conditions of Theorem \ref{thm:CoxCom} and therefore corresponds to a balanced complex $\Delta$ with class $D_\Delta = D$.
\end{proof}

\begin{remark}\label{rem:MonIrr}
Since the balancing condition depends strongly on the base ring $R$ (for instance, see Proposition \ref{prop:chardep}), this result illustrates an intricate characteristic-dependent behavior of the set of effective divisor classes on $\M_{0,n}$ that appears far from evident from a direct modular perspective.  We explore this in more detail below in the case of degree two divisors (cf., Theorem \ref{thm:CharTwo}).
\end{remark}

If $\Delta$ and $\Delta'$ are balanceable complexes and $\Delta' \subseteq \Delta$, then $D_\Delta = D_{\Delta'} + \sum E_I$ for some collection (possibly empty) of exceptional classes, by Lemma \ref{lem:DelClass}.  This motivates the following concept.

\begin{definition}
A balanceable complex $\Delta$ is \emph{minimal} if for any proper subset $\Delta' \subsetneq \Delta$, the complex $\Delta'$ is not balanceable.
\end{definition}

\begin{proposition}\label{prop:MinCyc}
Suppose $R$ is a field and $\Delta$ is a minimal complex.  Then the set of $R$-weightings that balance $\Delta$ (together with zero) forms a 1-dimensional vector space.
\end{proposition}

\begin{proof}
Write $\Delta = \{\sigma_1,\ldots,\sigma_m\}$.  The fact that $R$ has no zero-divisors implies that if $\{w_i\}$ is an $R$-weighting that balances $\Delta$, then $\{rw_i\}$ is also an $R$-weighting balancing $\Delta$, for any $r \in R\setminus\{0\}$.  Suppose there is a balancing $\{w'_i\}$ that is not a constant multiple of the balancing $\{w_i\}$.  The Laurent polynomials $f_{(\Delta,\{w_i\})}$ and $f_{(\Delta,\{w_i'\})}$ have the same terms, which are in bijection with $\Delta$.  Choose a simplex $\sigma_j\in\Delta$ and consider the collection of weights $\{w_j'w_i - w_jw_i'\}_{i=1}^m$.  By assumption these are not all zero, and the subset of nonzero elements balances the corresponding subset of $\Delta$ (cf., Remark \ref{rem:sum}).  Since this latter subset is contained in $\Delta \setminus \{\sigma_j\}$, we obtain a contradiction to the hypothesis that $\Delta$ is minimal.
\end{proof}

\begin{proposition}\label{prop:MinSys}
If $\Delta$ is a non-singular, minimal $(d-1)$-complex, with $d\le n-4$, then the set of balancing weights on $\Delta$ corresponds to the complete linear system $|D_{\Delta}|$.
\end{proposition}

\begin{proof}
Let $f\in |D_{\Delta}|$.  There are two cases we must rule out: (1) $f$ is divisible by some $x_I$, and (2) $f$ is not divisible by any $x_I$ but it corresponds to a balanced complex supported on a complex $\Delta' \ne \Delta$.   The second of these cannot occur due to Proposition \ref{prop:UniCom}, so suppose the first occurs.  The argument in the proof of Proposition \ref{prop:EffIff} shows that there is a balanceable complex $\Delta' \subsetneq \Delta$; but this contradicts the minimality hypothesis.
\end{proof}

\begin{corollary}\label{cor:UniSec}
If $\Delta$ is a non-singular, minimal $(d-1)$-complex over a field, with $d\le n-4$, then $h^0(\M_{0,n},D_\Delta) = 1$.
\end{corollary}

\begin{proof}
This follows immediately from Propositions \ref{prop:MinCyc} and \ref{prop:MinSys}.
\end{proof}

Minimality is helpful in ruling out one type of reducibility of elements in the monoid of effective divisor classes, namely, the existence of a decomposition $D = D' + \sum E_I$ where $\deg(D')=\deg(D)$.  Another, more subtle, form of reducibility is when $D = D_1 + D_2$ with $\deg(D_i) < \deg(D)$.  The following definition is aimed at studying this phenomenon.

\begin{definition}
For a $d_1$-complex $\Delta_1$ and a $d_2$-complex $\Delta_2$, their \emph{product} is the $(d_1+d_2+1)$-complex \[\Delta_1\cdot \Delta_2 = \{\sigma_1\cup \sigma_2~|~\sigma_i\in\Delta_i\}.\]
\end{definition}

\begin{remark}
If $\underline{\Delta}_1$, $\underline{\Delta}_2$ are balanced complexes, corresponding to polynomials $g_1,g_2\in\Cox(\M_{0,n})$, then it is not generally the case that $g_1g_2\in\Cox(\M_{0,n})$ corresponds to a balanced $R$-weighting on the complex $\Delta_1\cdot\Delta_2$, since cancellation may occur when expanding out this product.
\end{remark}

\begin{theorem}\label{thm:CoxGen}
Let $\Delta$ be a non-singular minimal $(d-1)$-complex, $d \le n-4$, over a field, and suppose there is no decomposition $\Delta = \Delta_1 \cdot \Delta_2$ with $\Delta_i$ minimal non-singular complexes satisfying $\Supp(\Delta_1) \cap \Supp(\Delta_2) = \varnothing$.  Then $D_\Delta$ is irreducible in $M(\M_{0,n})$ and every generating set for $\Cox(\M_{0,n})$ includes the unique (up to scalar) section of $D_\Delta$.
\end{theorem}

\begin{proof}
The hypotheses of Corollary \ref{cor:UniSec} are satisfied, so $|D_\Delta|$ is spanned by a single element, call it $f$, which is not divisible by any exceptional divisor section $x_I$.  This latter condition implies that there is no non-trivial expression of the form $D_\Delta = D+\sum E_I$ with $D$ effective.  Therefore, the only way $D_\Delta$ could be reducible is if we have a decomposition $D_\Delta = D_1 + D_2$ with $\deg(D_i) < D$.  Such a decomposition would imply a factorization $f = f_1f_2$.  Since $f$ is not divisible by any $x_I$, neither are the $f_i$, so they correspond to balanced complexes $\underline{\Delta}_1$ and $\underline{\Delta}_2$.  We claim that the underlying complexes $\Delta_1$ and $\Delta_2$ are supported on disjoint subsets of $[n-1]$: \[\Supp(\Delta_1) \cap \Supp(\Delta_2) = \varnothing.\]  Indeed, suppose $\ell\in \sigma_1\cap\sigma_2$ for some index $\ell\in [n-1]$ and simplices $\sigma_i \in \Delta_i$.  Let $k_i \ge 1$ be the largest exponent of $y_\ell$ appearing in the Laurent polynomial $f_{\underline{\Delta}_i}$.  It suffices, by the non-singularity hypothesis on $\Delta$, to show that the expansion of $f_1f_2$ has a nonzero term divisible by $y_\ell^{k_1+k_2}$; but this is clear, so the claim holds.  It follows that there is no cancellation of terms when expanding the product $f_1f_2$, so $\Delta_1 \cdot \Delta_2 = \Delta$ and each $\Delta_i$ is non-singular.  Moreover, for this same reason if either $\Delta_i$ were not minimal, then $\Delta$ would not be minimal, so we obtain a contradiction to the hypothesis that $\Delta$ is minimal and not a product of disjointly supported minimal non-singular complexes.

Finally, we note that in general the monoid $M(X)$ is generated by the classes of a generating set of the ring $\Cox(X)$; since this monoid is minimally generated by its irreducible elements, we see that any generating set for the Cox ring must include a section of each irreducible divisor class.
\end{proof}


\section{Catalogue of divisor classes}

In this section we continue to work over an arbitrary ring $R$, unless otherwise stated, and we introduce various important families of balanceable complexes, with a view towards understanding minimal complexes and irreducible divisor classes (cf., Theorem \ref{thm:CoxGen}).  A basic but important observation is that in a balanceable complex, every face of every simplex must belong to more than one simplex.  This restriction, as we shall see below, means that minimal complexes tend to be triangulations of closed manifolds, or closely related objects.  We start with a simple example.

\begin{example}\label{ex:orthoplex}
A minimal 0-complex is a pair of vertices; the associated divisor class is a degree one boundary divisor, i.e., the strict transform of a hyperplane in $\PP^{n-3}$.  This corresponds to an invariant $\frac{y_i}{z_i} - \frac{y_j}{z_j}$, which upon clearing denominators yields a binomial with class $H - \sum_{I \not\ni i,j}E_I$.  Recall that the orthoplex, or cross-polytope, is the convex hull in $\mathbb{R}^d$ of the $2d$ vectors $\pm e_i$.  For example, the 1-orthoplex is an interval, the 2-orthoplex is a square, and the 3-orthoplex is an octahedron.  All of the facets of the orthoplexes are simplices.  We claim that the $d$-fold product of minimal 0-complexes with disjoint support is a minimal $(d-1)$-complex comprising the $2^d$ facets of the $d$-orthoplex; the associated divisor class is a sum of $d$ degree one boundary divisor classes.  Indeed, the disjoint support hypothesis implies that clearing denominators commutes with taking the product of the $d$ binomials of the form $\frac{y_i}{z_i} - \frac{y_j}{z_j}$, so the simplices of the product complex correspond simply to the terms of this expansion.  This complex is minimal, since removing any simplices would yield a face contained in a single simplex.
\end{example}


\subsection{Degree two divisors}\label{sec:Deg2}

Here we investigate in more detail the case of 1-complexes and their corresponding quadric divisor classes.  For simplicity we work over an arbitrary field $k$.  Recall that 1-complexes are simply graphs without multiedges.  This case is special because one can restrict attention entirely to non-singular complexes, or in other words, loop-free graphs:

\begin{proposition}\label{prop:loopfree}
If a balanceable 1-complex $\Delta$ is singular, then its class $D_\Delta$ is a sum of boundary divisor classes; in particular, it is reducible.
\end{proposition}

\begin{proof}
A singular 1-complex has a simplex of the form $\sigma = \{i,i\}$, so by Lemma \ref{lem:DelClass} the class $D_\Delta$ involves only $E_I$ with $i\notin I$.  Thus we can identify $D_\Delta$ with an effective divisor class on the toric variety $X_{n-1}$ obtained by blowing up only coordinate points in $\PP^{n-3}$.  The effective cone of this toric variety is generated by exceptional divisors and strict transforms of hyperplanes, so $D_\Delta$ is a sum of exceptional divisors and two strict transforms of hyperplanes.  Returning to $\M_{0,n}$, this shows that $D_\Delta$ is a sum of degree zero boundary divisors and two degree one boundary divisors.
\end{proof}

\begin{remark}
The above proof shows that for a balanceable $(d-1)$-complex $\Delta$ for any $d \ge 2$, the class $D_\Delta$ is a sum of exceptional boundary divisors and $d$ degree one boundary divisors if $\Delta$ contains a simplex $\sigma$ with $|\Supp(\sigma)| = 1$.  Of course, for $d \ge 3$ there are many singular complexes that do not contain such maximally singular simplices, and for these complexes $D_\Delta$ may in fact be irreducible (cf., \S\ref{sec:CTasDel}).
\end{remark}

Already with non-singular 1-complexes we find some intriguing behavior.

\begin{theorem}\label{thm:CharTwo}
The set of effective classes in $\Pic ( \M_{0,n} )$ depends on $\Char{k}$, for all $n \geq 7$.
\end{theorem}

\begin{proof}
Consider a 1-complex $\Delta$ on $\{1,\ldots,6\}$ given by the edges of a disjoint union of two triangles.  This is easily seen to be balanceable if and only if $\Char{k} = 2$ (or more generally, for an arbitrary base ring $R$, if and only if $2\in R$ is a zero-divisor).  Since any subcomplex is not balanceable, the result follows from Proposition \ref{prop:EffIff}.
\end{proof}

\begin{remark}
The preceding result does not imply that the effective cone $\Eff(\M_{0,n})$ depends on $\Char{k}$.  Indeed, although $D_\Delta$ is effective if and only if $\Char{k} = 2$ when $\Delta$ is a disjoint union of two triangles, we shall show in \S\ref{subsec:twotri} that $2D_\Delta$ is effective over any field.
\end{remark}

Graph-theoretic cycles, such as the triangles used in the proof of Theorem \ref{thm:CharTwo}, play an important role.  We call a cycle \emph{even} or \emph{odd}, respectively, if it has an even or odd number of edges.  A loop $\sigma = \{i,i\}$ is a cycle with one edge, so it is odd.

\begin{proposition}
\label{Prop:MinimalGraphs}
The minimal 1-complexes over $k$ are the following (see Figure \ref{fig:minimal}):
\begin{enumerate}
\item  an even cycle;
\item  a union of two odd cycles, meeting in a single vertex, and
\item[(3a)]  a disjoint union of two odd cycles, if $\Char{k} = 2$.
\item[(3b)] a union of two odd cycles, connected by a chain of one or more edges, if $\Char{k} \ne 2$.
\end{enumerate}
\end{proposition}

\begin{figure}\begin{center}
\scalebox{0.85}{\input{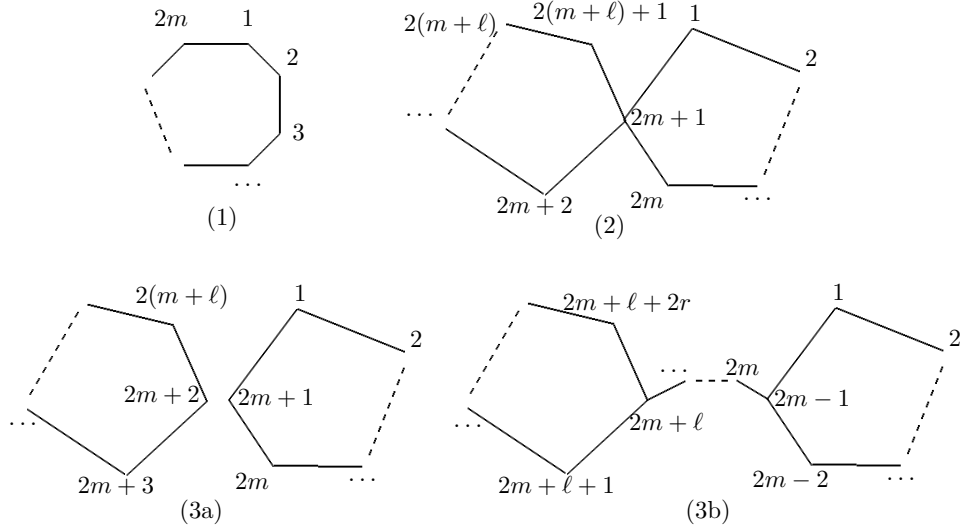}}
\caption{The minimal 1-complexes in Proposition \ref{Prop:MinimalGraphs}.}
\label{fig:minimal}
\end{center}\end{figure}

\begin{proof}
If a graph contains no cycles, then it contains a valence-one vertex and hence cannot be balanced, so every balanceable graph contains a cycle.  Note that an even cycle is balanceable whereas an odd cycle is not; thus, if a minimal graph contains more than one cycle, they must all be odd.  If there are two odd cycles that share one or more edges, then by deleting the edges they have in common, we obtain a disjoint union of cycles whose total number of edges is even.  It follows that the graph contains, and is therefore equal to, a complex as in (1), (2), (3a), or (3b).  Thus, in a minimal graph with more than one cycle, no two cycles can share an edge.  Note that if $\Char{k} = 2$ then a disjoint union of two odd cycles is balanceable, but if a balanced graph is disconnected when $\Char{k} \ne 2$, then each of the connected components is balanced.  In this case, the two cycles must be connected by a chain of edges.
\end{proof}

\begin{theorem}\label{thm:GenKV}
A degree two divisor class is irreducible in $M(\M_{0,n})$ if and only if it is of the form $D_\Delta$ where $\Delta$ is one of the loop-free graphs described in Proposition \ref{Prop:MinimalGraphs}, except for the square.
\end{theorem}

\begin{proof}
Suppose $\Delta$ is one of the graphs mentioned in the hypothesis.  Note that this forces $n \ge 6$.  Since $\Delta$ is non-singular and minimal, and $2 \le n-4$, to show that $D_\Delta$ is irreducible it suffices by Theorem \ref{thm:CoxGen} to show that $\Delta$ is not a product of disjointly supported non-singular minimal 0-complexes.  But such a product is always a square (cf., Example \ref{ex:orthoplex}), so this indeed holds.

Conversely, let $D$ be an irreducible degree two divisor class.  Every class of the form $D-E_I$ is not effective, by irreducibility, so Corollary \ref{cor:AllClass} implies that $D = D_\Delta$ for some balanceable 1-complex $\Delta$.  Proposition \ref{prop:loopfree} implies that $\Delta$ is non-singular, and it cannot be a square since the divisor class corresponding to a square is reducible (cf., Example \ref{ex:orthoplex}), so it only remains to show that $\Delta$ is minimal.  If there were a balanceable subcomplex $\Delta' \subsetneq \Delta$, then it must be non-singular since $\Delta$ is, so again using that $2 \le n-4$ we see from the formula in Lemma \ref{lem:DelClass} that there is a non-trivial expression $D_\Delta = D_{\Delta'} + \sum E_I$, thus contradicting the irreducibility hypothesis.
\end{proof}

By Corollary \ref{cor:UniSec}, the divisor class associated to each graph in Theorem \ref{thm:GenKV} has a unique section, up to scalar.  Since these classes are irreducible, any set of generators for $\Cox(\M_{0,n})$ must include these sections.  However, not all of these classes are needed to generate the effective cone $\Eff(\M_{0,n})$; indeed, some of these classes span extremal rays (as we shall see in \S\ref{subsec:oct}), whereas others span non-extremal rays, as the following example illustrates.

\begin{example}\label{ex:2Dboundary}
Let $G$ be the 1-complex in Proposition \ref{Prop:MinimalGraphs} of type (3b) shown in Figure \ref{fig:tribritri}.  We claim  $2D_G$ is an effective sum of multiple distinct boundary classes.  Label the vertices as in Figure \ref{fig:tribritri}.  By clearing denominators of \[(\frac{y_1}{z_1} - \frac{y_4}{z_4})(\frac{y_2}{z_2} - \frac{y_4}{z_4})(\frac{y_3}{z_3} - \frac{y_5}{z_5})(\frac{y_3}{z_3} - \frac{y_6}{z_6}),\] we obtain a class that is the sum of four degree one boundary divisors.  A straightforward calculation shows $2D_G$ is this class plus exceptional boundary divisors.
\end{example}

\begin{figure}\begin{center}
\scalebox{0.9}{\input{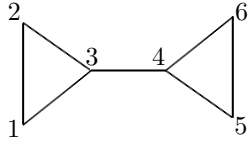}}
\caption{The balanceable graph in Example \ref{ex:2Dboundary}.}
\label{fig:tribritri}
\end{center}\end{figure}

This implies the following:

\begin{corollary}\label{Cor:NotJustExtremal}
If $n \ge 7$ and $\Char(k)\ne 2$, then there exist elements of $\Cox(\M_{0,n})$ outside the subring generated by sections of the extremal rays of $\Eff(\M_{0,n})$.
\end{corollary}


\subsection{Degree three divisors}\label{sec:Deg3}

We now return to a general base ring $R$ and begin to investigate balanceable 2-complexes.  A complete classification of minimal complexes and irreducible divisor classes, as we provided for 1-complexes, seems combinatorially out of reach.  Indeed, as we illustrate in this section through a collection of examples, there is a rich and elegant geometry, often quite classical yet intricate, underlying minimality of 2-complexes.  We saw in \S\ref{sec:Deg2} that minimal 1-complexes can be viewed as triangulations of a circle or pair of circles; here the key construction appears to be triangulations of the sphere (or other closed surfaces) and ways to attach these.

\subsubsection{Bipyramids}

We saw in Theorem \ref{thm:CoxGen} that under suitable hypotheses, the divisor class associated to a non-singular minimal complex is irreducible \emph{except} when that complex is a product of disjointly supported non-singular minimal complexes.  Thus, it is important to identify such products.  For a 2-complex, the only decomposition into a product is given by a 1-complex and a 0-complex.  But minimal 0-complexes are simply pairs of vertices, and we classified minimal 1-complexes in Proposition \ref{Prop:MinimalGraphs}.  The product of a 1-complex $\Delta$ with a disjointly supported minimal 0-complex is a discrete analogoue of the suspension construction in topology.  For instance, if $\Delta$ is an even cycle, say a $(2m)$-gon, we obtain the faces of a $(2m)$-gonal bipyramid.  If $\Delta$ is the disjoint union of an $m_1$-gon and an $m_2$-gon, then we get the faces of $m_i$-gonal bipyramids attached at their north and south poles.

\subsubsection{Cycle of tetrahedra}\label{subsec:CycTet}

Consider a cycle of $m$ tetrahedra attached as follows (see Figure \ref{fig:TetCyc}): \[ \{1,2,3,4\},\{3,4,5,6\},\ldots,\{2m-1,2m,1,2\}.\]  We claim that the collection of all faces of these tetrahedra forms a balanceable, and in fact minimal, 2-complex.  Indeed, if for each tetrahedron $\{a,b,c,d\}$ we weight the faces $w_{abc} = w_{abd} = 1$ and $w_{acd} = w_{bcd} = -1$, then the balancing conditions are easily seen to hold; minimality follows from the observation that removing any triangles results in an edge contained in only one triangle.

\begin{figure}\begin{center}
\scalebox{1.0}{\input{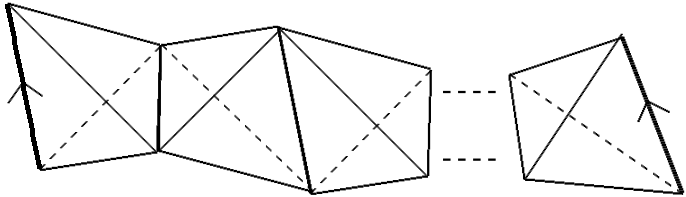}}
\caption{A cycle of tetrehedra.}
\label{fig:TetCyc}
\end{center}\end{figure}

\subsubsection{Triangulated octahedron}

We saw above that the faces of an octahedron form a minimal complex, but the associated divisor class is uninteresting since the octahedron is the suspension of a square, so this complex decomposes into a product of minimal 0-complexes.  However, we can modify this 2-complex to produce one that is minimal yet no longer decomposes into such a product.  For instance, the triangulation of the faces of the octahedron pictured in Figure \ref{fig:Octahedron} is readily seen to be balanced with weights in $\{\pm 1\}$.  Minimality is clear since, as usual, removing triangles yields edges contained in a single triangle.  Since this non-singular complex has $18$ vertices, we therefore obtain by Theorem \ref{thm:CoxGen} an irreducible divisor class, hence Cox ring generator, over any field, for $\M_{0,n}$ with $n \ge 19$.  Indeed, since this 2-complex is not the suspension of a graph it does not decompose into a product of disjointly supported non-singular minimal complexes.

\begin{figure}
\begin{center}
\includegraphics[scale=0.4]{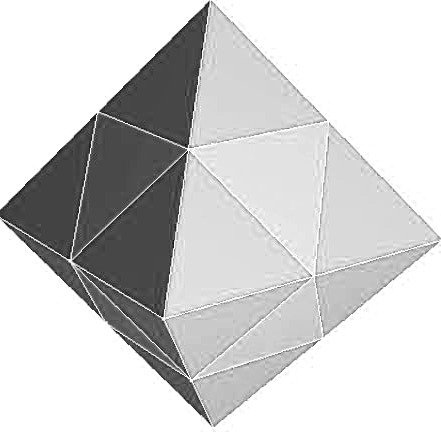}
\end{center}
\caption{Triangulated surface of an octahedron.}
\label{fig:Octahedron}
\end{figure}

\subsubsection{Triangulated cube}

By triangulating the six faces of a cube into four triangles each, as depicted in Figure \ref{fig:Cube}, we obtain a minimal 2-complex that can be balanced with weights in $\{\pm 1\}$ and which is not the product of minimal complexes.  The analysis is quite similar to the triangulated octahedron described above.  Here we obtain an irreducible divisor class on $\M_{0,n}$ for all $n \ge 15$.

\begin{figure}
\begin{center}
\includegraphics[scale=1.0]{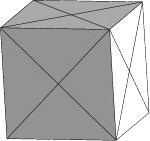}
\end{center}
\caption{Triangulated surface of a cube.}
\label{fig:Cube}
\end{figure}

\subsubsection{Triangulated truncated octahedron}

The truncated octahedron has 8 hexagonal faces and 6 square faces (see Figure \ref{fig:Truncatedoctahedron}).  By subdividing each square face into 4 triangles as before, and each hexagonal face into 6 triangles meeting at a central vertex, we again obtain a minimal complex with irreducible divisor class, now defined on $\M_{0,n}$ for $n \ge 39$.

\begin{figure}
\begin{center}
\includegraphics[scale=0.25]{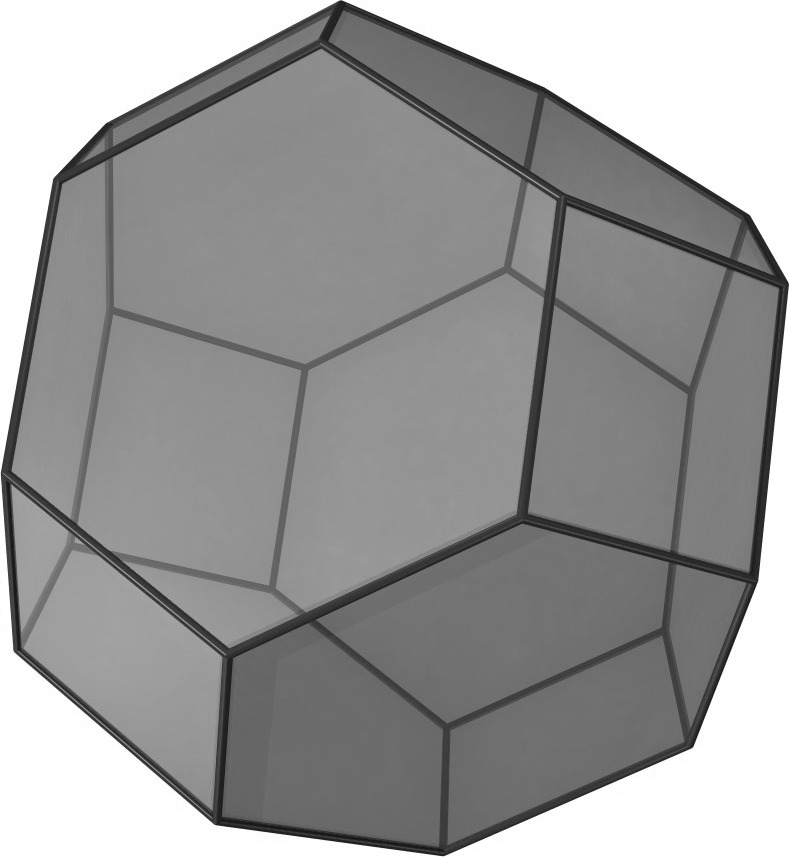}
\end{center}
\caption{Truncated octahedron.}
\label{fig:Truncatedoctahedron}
\end{figure}

\subsubsection{Triangulation of the torus}

The sphere does not seem to be special with regard to triangulations yielding minimal complexes.  For instance, we can triangulate the torus as shown in Figure \ref{fig:Torus} to obtain an irreducible divisor class on $\M_{0,n}$ whenever $n \ge 10$.

\begin{figure}
\begin{center}
\includegraphics[scale=0.4]{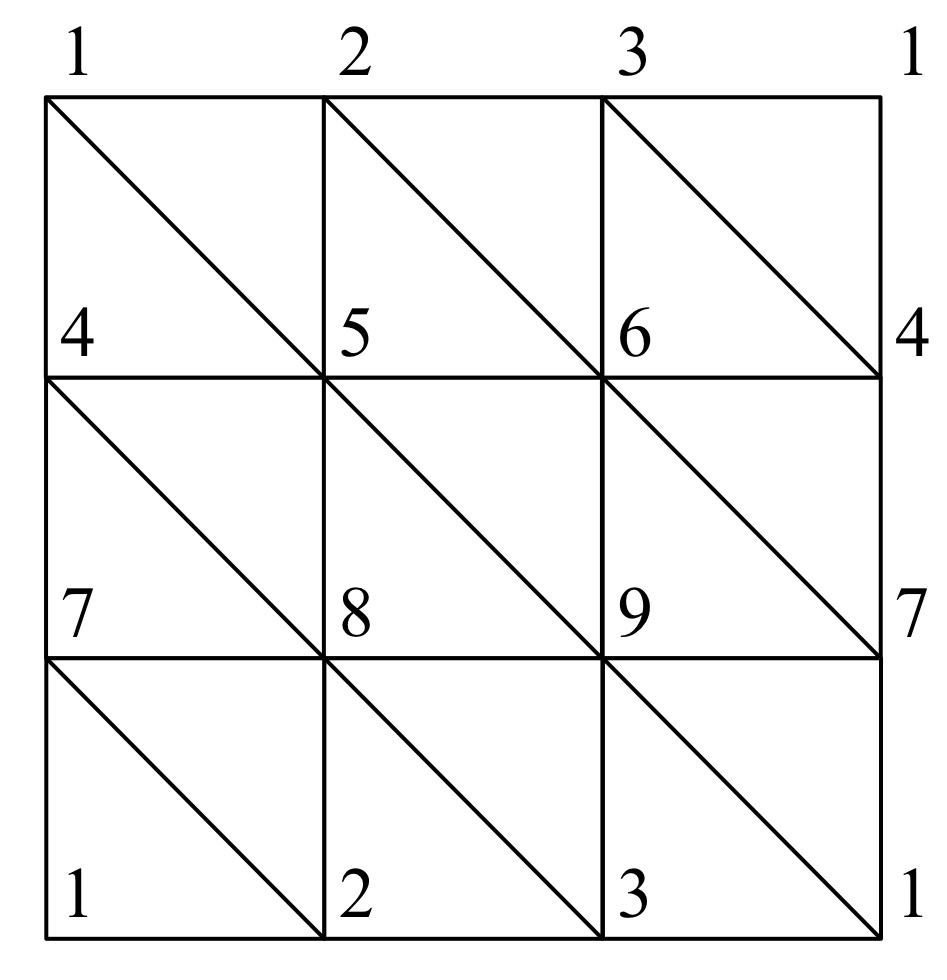}
\end{center}
\caption{Triangulated torus.}
\label{fig:Torus}
\end{figure}

\subsubsection{And many more}

The examples here of non-singular minimal 2-complexes comprise only a tiny fraction of what exists.  It appears that most triangulations of a closed surface such that each vertex belongs to an even number of triangles yield minimal complexes balanced by weights in $\{\pm 1\}$.  Given such a triangulation, it is usually possible to refine the triangulation to obtain another minimal complex supported on a larger number of vertices.  In this way, as $n$ increases the number of minimal 2-complexes yielding irreducible divisor classes on $\M_{0,n}$ seems to grow quite quickly.  A classification of surface triangulations yielding minimal complexes, or an estimate of the growth rate of such objects, would be interesting.  Moreover, the example of a cycle of tetrahedra suggests that minimal complexes are sometimes obtained by attaching various closed surface triangulations in an appropriate way.  Additionally, all the examples described in this section are non-singular and minimal over any ring $R$, with weights in $\{\pm 1\}$; surely there are interesting characteristic dependent complexes as we saw in the case of graphs (cf., Proposition \ref{Prop:MinimalGraphs}), and interesting singular complexes too (in fact, many hypertree divisors correspond to singular complexes, cf., \S\ref{sec:CTasDel}).


\section{Relation to hypertree divisors}\label{sec:Hypertree}

In this section we first recall the definition of a hypertree divisor from \cite{CT13} and then prove that the only ones with degree two for some choice of $\psi$-class are those coming from $n=6$ and $n=7$.  We proceed to identify balanced complexes corresponding to a few hypertree divisors and then present our counterexamples to the hypertree conjecture.

\subsection{Background to the hypertree construction}\label{sec:CTbackground}

The divisors constructed by Castravet and Tevelev are parameterized by a combinatorial object they term a \emph{hypertree}.

\begin{definition}
An \emph{irreducible hypertree} on a set $N$ is a collection of subsets $\Gamma = \{ \Gamma_1 \ldots , \Gamma_d \}$ of $N$ satisfying the following properties:
\begin{enumerate}
\item  Each subset $\Gamma_i$ contains at least 3 elements.
\item  Any $j \in N$ is contained in at least two subsets $\Gamma_i$.
\item  (Convexity)  For any $S \subset \{ 1, \ldots , d \}$ with $1 < \vert S \vert < d$,
$$ \vert \bigcup_{i \in S} \Gamma_i \vert -2 > \sum_{i \in S} ( \vert \Gamma_i \vert -2) . $$
\item  (Normality)
$$ \vert N \vert -2 = \sum_{i=1}^d ( \vert \Gamma_i \vert -2) .$$
\end{enumerate}
\end{definition}

For each hypertree they consider the product of forgetful maps
$$ \overline{M}_{0,n+1} \to \prod_{i=1}^d \overline{M}_{0,\Gamma_i \cup \{ n+1 \}} .$$
The hypertree axioms precisely guarantee that this map is birational.  For example, (1) implies that none of the forgetful maps are trivial, and (4) implies that the dimensions of the domain and codomain are equal.  It turns out this map always contracts a single, irreducible, non-boundary divisor, which Castravet and Tevelev call a \emph{hypertree divisor}.  By construction, each of these divisors is the pull-back of an irreducible divisor from $\overline{M}_{0,n}$, and each generates an extremal ray of the effective cone.

\begin{remark}\label{rem:DetForm}
There is a determinantal formula for hypertree divisors with all $\Gamma_i$ of size three in terms of homogeneous coordinates on $(\PP^1)^n$ in \cite[Proposition 8.1]{CT13}.  Their formula, and proof of its validity, extend immediately to the case that at most one $\Gamma_i$ has cardinality four\footnote{In fact, Opie has provided a determinantal expression for \emph{all} hypertrees in \cite[Theorem 3.1]{Opi13}.}, which covers all hypertrees for $n \le 9$ according to \cite[Figure 1]{CT13}.  Moreover, the translation from this coordinate system to ours on $\PP^{n-3}$ is straightforward.  In this way, one can write explicitly the element of $\Cox(\M_{0,n})$, hence balanced complex, corresponding to these hypertrees.
\end{remark}

\subsection{Degree of hypertree divisors}\label{sec:Deg}

Here we show that, up to permutation and pull-back, there are only two hypertree divisors of degree two in any Kapranov presentation $\M_{0,n} \cong \Bl\PP^{n-3}$.  These are the unique hypertree divisors in $n=6$ and $n=7$.  First, some preliminary notation and results.

\begin{definition}
The \emph{degree of $D\in\Pic(\M_{0,n})$ relative to $\psi_i$} is the coefficient of $H$ under the isomorphism $\Pic(\M_{0,n}) \cong \mathbb{Z}H\bigoplus_I \mathbb{Z}E_I$ induced by the corresponding isomorphism $\M_{0,n} \cong \Bl \PP^{n-3}$.  We denote this by $\deg(D)$, or $\deg^i(D)$ if the dependence on $\psi_i$ is to be emphasized.  The \emph{minimal degree} is $\deg^{min}(D) := \min_{1\le i \le n}\{\deg^i(D)\}$.
\end{definition}

\begin{remark}
This can be phrased without mention of Kapranov's blow-up construction.  Fix $n-1$ general lines in $\PP^2$, indexed by $\{1,\ldots,\widehat{i},\ldots,n\}$, and a point $p\in \PP^2$ not lying on any of them.  The pencil of lines through $p$ yields a curve $C_i \subset \M_{0,n}$ by viewing $p$ as the $i^{\text{th}}$ marked point and the intersection with the lines as the remaining marked points.  Then $\deg^i(D) = D\cdot C_i$.  To see this, note that
\[ \Delta_I \cdot C_i = \begin{cases}
1 & \textrm{if $\vert I \vert = 2$, $i \notin I$}\\
0 & \textrm{otherwise} \end{cases}, \]
where $\Delta_I$ denotes the boundary divisor with points indexed by $I$ on one side of a node.  The result then follows from the standard formula for the $\psi$ classes (see, for example, \cite[Exercise 1.35]{Morrison}):
$$ \psi_i = \sum_{q,r \in I, i \notin I} \Delta_I \text{ for any fixed distinct } q,r \neq i . $$
\end{remark}

In terms of our coordinatization of $\Cox(\M_{0,n})$, which depends on a choice of $\psi$-class, the degree of the class of a multi-homogeneous polynomial is the total degree in the $y$-variables.

\begin{example}
If $D$ is effective then $\deg(D) \ge 0$.  The degree of a boundary divisor is 1 or 0, depending on the choice of $\psi$-class, so the minimal degree of any boundary divisor is 0.  The minimal degree of any Keel-Vermeire divisor is 2.
\end{example}

\begin{proposition}
\label{Prop:DegreeOne}
Every effective divisor class of minimal degree 0 or 1 is an effective combination of boundary divisor classes.
\end{proposition}

\begin{proof}
Suppose $\deg^{min}(D) \le 1$, so that for some Kapranov isomorphism we have $D = aH - \sum_I b_I E_I$ with $a=0$ or 1.  Let $L_I$ be the strict transform of a line in $\PP^{n-3}$ passing through a general point of the linear span of the points in $I$.  Since $L_I$ covers $\overline{M}_{0,n}$, we see that $L_I \cdot D' \geq 0$ for all effective divisors $D'$.  Thus, if $a=0$, we must have $b_I \leq 0$ for all $I$, and hence $D$ is an effective sum of the boundary divisors $E_I$.  If $a=1$ then $D$ is the sum of an irreducible effective divisor of degree 1 and an effective divisor of degree 0.  Every irreducible effective divisor of degree 1, however, is the strict transform of a hyperplane in $\PP^{n-3}$.  Since such a hyperplane passes through at most $n-3$ general points, and any hyperplane passing through a given set of points also passes through the linear subspaces spanned by these points, we see that $D$ is an effective combination of the boundary divisor $H - \sum_{i,j \notin I} E_I$, for some $i$ and $j$, and exceptional boundary divisors.
\end{proof}

\begin{lemma}
\label{Lemma:HypertreeDegree}
Let $D_{\Gamma}$ be the pull-back to $\M_{0,n}$ of an irreducible hypertree divisor in $\overline{M}_{0,I}$.  Then
\[ \deg^i(D_{\Gamma}) = \begin{cases}
d-v_i & \textrm{if $i \in I$}\\
d-1 & \textrm{if $i \notin I$} \end{cases}\]
where $d$ is the number of sets in $\Gamma$ and $v_i$ is the valence of vertex $i$.  In particular, $\deg^{min}D_{\Gamma} = d-v_{max}$, where $v_{max}$ is the maximum valence of the vertices in $\Gamma$.
\end{lemma}

\begin{proof}
If $i \notin I$, the result follows immediately from \cite[Theorem 4.2]{CT13}.  On the other hand, if $i \in I$, then, for the choice of Kapranov isomorphism given in \cite[Theorem 4.2]{CT13}, we have $\deg^i(H) = n-2$, $\deg^i(E_j) = 1$ if $j \neq i$, and $\deg^i(E_J) = 0$ for all other sets $J$.  It follows then that
$$ \deg^i(D_{\Gamma}) = (d-1)(n-2) - \sum_{j \neq i} (d-v_j ) = 2-d-n + \sum_{j \neq i} v_j $$
$$ = 2-d-n-v_i + \sum_{j=1}^d \vert \Gamma_j \vert = d-v_i . $$
The last statement follows from the fact that every vertex has valence $\ge 2$.
\end{proof}

\begin{proposition}\label{Prop:degtwoCT}
The only irreducible hypertree divisors of minimal degree 2 are the unique hypertree divisor on 6 vertices, the unique hypertree divisor on 7 vertices, and their pull-backs.
\end{proposition}

\begin{proof}
Note that, by Proposition \ref{Prop:DegreeOne}, every irreducible hypertree divisor has minimal degree at least 2.  By Lemma \ref{Lemma:HypertreeDegree}, it therefore suffices to consider a hypertree with a vertex $v$ of valence $d-2$.  Suppose without loss of generality that $\Gamma_1$ and $\Gamma_2$ are the only sets in the hypertree that do not contain $v$. By convexity, any pair of sets cannot intersect in more than one vertex, so for each vertex $w \neq v$ there is at most one set containing both $v$ and $w$.  Hence
$$ \sum_{\Gamma_i \ni v} ( \vert \Gamma_i \vert - 2 ) = \vert \bigcup_{\Gamma_i \ni v} \Gamma_i \vert + (d-3) - 2(d-2) = \vert \bigcup_{\Gamma_i \ni v} \Gamma_i \vert - (d-1). $$
By normality, this implies that
$$ \vert \Gamma_1 \vert + \vert \Gamma_2 \vert = n+d+1 - \vert \bigcup_{\Gamma_i \ni v} \Gamma_i \vert .$$
Since the valence of each vertex is at least 2, every $w \neq v$ must be contained in either $\Gamma_1$ or $\Gamma_2$, and since $|\Gamma_1 \cap \Gamma_2| \le 1$, it follows that $\vert \Gamma_1 \vert + \vert \Gamma_2 \vert$ is either $n-1$ or $n$.  Thus we have
\[\vert \bigcup_{\Gamma_i \ni v} \Gamma_i \vert \geq \begin{cases}
n & \textrm{if $\vert \Gamma_1 \vert + \vert \Gamma_2 \vert = n-1 $}\\
n-1 & \textrm{if $\vert \Gamma_1 \vert + \vert \Gamma_2 \vert = n$}
\end{cases}\]
so $n=d+2$ or $d+1$, but the latter is impossible by normality.  This forces $\vert \Gamma_i \vert = 3$ for all $i$, so $\vert \Gamma_1 \vert + \vert \Gamma_2 \vert = 6$.  But again, this sum is either $n-1$ or $n$, so $n$ is either 6 or 7.
\end{proof}

\subsection{Simplicial complex presentation of hypertree divisors}\label{sec:CTasDel}

Using the method described in Remark \ref{rem:DetForm}, it is straightforward to compute that the minimal 1-complex comprising the edges of a triangle meeting a triangle at a single vertex is the $n=6$ hypertree divisor (i.e., the Keel-Vermeire divisor), and the minimal 1-complex comprising the edges of a hexagon is the $n=7$ hypertree divisor.  It follows, then, from  Proposition \ref{Prop:degtwoCT} that all the other minimal 1-complexes described in Theorem \ref{thm:GenKV} yield irreducible divisor classes that are not equal to that of a hypertree divisor.

Note that, by Lemma \ref{Lemma:HypertreeDegree}, in order to get a degree 2 divisor class from the $n=7$ hypertree we must pick the $\psi$-class corresponding to the central vertex in its depiction in \cite[Figure 1]{CT13}.  Another straightforward computation using the method of Remark \ref{rem:DetForm} shows that if we use any other $\psi$-class for this hypertree we get the degree 3 divisor corresponding to a cycle of 3 tetrahedra as described in \S\ref{subsec:CycTet}.  Many of the $n=8$ and $n=9$ hypertrees correspond to singular, though still quite interesting, simplicial complexes.  A systematic translation between the hypertree language and the balanced complex language seems a worthwhile future pursuit.  Since the hypertree combinatorics are rooted in the geometry of products of forgetful maps, reinterpreting that geometry in this setting may be a natural place to start.

\subsection{Counterexamples to the hypertree conjecture}\label{sec:CTcounter}

Here we provide two counterexamples to the hypertree conjecture, one in $\M_{0,7}$ and one in $\M_{0,9}$.  Note that, by pulling these divisors back via forgetful maps, one obtains counterexamples in $\M_{0,n}$ for all $n \geq 7$.  In addition to the examples described below, we have checked by computer that, if $\Delta$ is the edges of a triangle meeting a pentagon at one vertex, then $D_\Delta \in \Eff(\M_{0,8})$ lies outside the cone generated by hypertree and boundary divisors.  We suspect that the divisor $D_\Delta$, where $\Delta$ is any even cycle with at least six edges, should generate an extremal ray of $\Eff(\M_{0,n})$, but at present we have no proof.

Our approach is via covering curves.  Recall that a curve $C \subset X$ is said to \emph{cover} a subvariety $Y \subset X$ if $C \subset Y$ and there is an irreducible curve numerically equivalent to $C$ passing through a general point of $Y$.  It is well-known that a divisor $D \subset X$ spans an extremal ray of the effective cone if $D$ is covered by a curve $C$ such that $D \cdot C < 0$ (see, e.g., \cite[Lemma 1.4.2]{Rulla}).

\subsubsection{The Case $n=9$}\label{subsec:oct}

Consider the divisor $D_{oct} := D_\Delta$ where $\Delta$ is the set of edges of an octagon.  We saw in Theorem \ref{thm:CoxGen} that $h^0 ( \M_{0,n} , D_{oct} ) = 1$, spanned by an irreducible section.  Our goal is to construct a curve $F_9$ that covers this divisor and satisfies $F_9 \cdot D_{oct} < 0$.

Fix a conic $C_9 \subset \PP^2$, points $p_1 , \ldots, p_7 \in C_9$, and a point $p_8 \in \PP^2 \smallsetminus C_9$.  We define a configuration of four lines and four conics in $\PP^2$ as follows:  for $1 \leq i \leq 7$, $i$ odd, let $L_i$ be the line through $p_2$ and $p_i$.  For $2 \leq i \leq 8$, $i$ even, let $C_i$ be the conic through $p_4$, $p_6$, $p_8$, $p_{i+3}$, and $p_{i-3}$, where here $i+3$ and $i-3$ are taken mod 8.

The pencil of lines through $p_6$ determines a curve in $\M_{0,9}$, where each line is labelled by its intersection with the lines and conics.  More specifically, if we let $S$ be the blow-up of $\PP^2$ at the intersection points of the lines and conics, then the strict transforms $\tilde{L_i}, \tilde{C_i}$ of the lines $L_i$ and the conics $C_i$ give disjoint sections of the map $\pi_{p_6} : S \to \PP^1$.  We write $F_9 \subset \M_{0,9}$ for the corresponding curve in the moduli space.  To compute the intersection number $F_9 \cdot H$, we note that $H = \psi_9$, hence if $\sigma_9 : \PP^1 \to S$ is the corresponding section, then by push-pull and \cite[Exercise 1.33]{Morrison} we have
$$ F_9 \cdot H = F_9 \cdot \psi_9 = \sigma_9^* (c_1 (\omega_{\pi_{p_6}})) = \tilde{C_9} \cdot \omega_{\pi_{p_6}} = - \tilde{C_9}^2 = -(2^2 - 7) = 3.$$
To compute the intersection numbers $F_9 \cdot E_I$, we simply count the number of singular fibers of the map $\pi_{p_6} : S \to \PP^1$.  Note that a fiber of this map is singular if and only if the corresponding line through $p_6$ passes through one of the points of intersection of the lines $L_i$ and conics $C_i$. By construction, therefore, we have
\[F_9 \cdot E_I = \begin{cases}
2 & \textrm{if } I = \{ 1,3,5,7 \}\\
1 & \textrm{if } I = \{ 2,4,6,8 \} \textrm{ or } I = \{ i,i+3,i-3 \} \textrm{ for $i$ odd}\\
0 & \textrm{otherwise.}\\
\end{cases}\]
It follows that
$$ F_9 \cdot D_{oct} = 2 \cdot 3 - 5 - 2 = -1. $$

\begin{proposition}
The curve $F_9$ covers $D_{oct}$.
\end{proposition}

\begin{proof}
Since $F_9 \cdot D_{oct} < 0$, any effective curve numerically equivalent to $F_9$ is contained in $D_{oct}$.  Since $D_{oct}$ is irreducible, it therefore suffices to show that $F_9$ covers a 5-dimensional space.  We will do this by showing that curves equivalent to $F_9$ dominate $M_{0,8}$.

More specifically, fix a line $L \subset \PP^2$ and points $q_1 , \ldots , q_7$ and $q_9$ in $L$.  We will show that there is a curve equivalent to $F_9$ whose image in $M_{0,8}$ under the map forgetting the 8th point contains the curve $L$ marked by the points $q_i$.  To see this, first choose a conic $C_9$ containing $q_9$, and label the other point of the intersection $L \cap C_9$ by $p_6$.  Choose a point $p_2 \in C_9$ and for $1 \leq i \leq 7$, $i$ odd, let $L_i$ be the line through $p_2$ and $q_i$.  Label the other point of the intersection $L_i \cap C_9$ by $p_i$.

Let $U\subset \PP^2$ be the complement of the points $p_i$ and $q_i$.  For any point $p \in U$, there exists a unique conic $C_2$ containing $\{ q_2, p_5, p_6, p_7, p \}$.  Similarly, there exists a unique conic $C_4$ containing $\{ q_4, p_1, p_6, p_7, p \}$, and a unique conic $C_6$ containing $\{ q_6, p_1, p_3, p_6, p \}$.  We claim that the set $Y \subset U \subset \PP^2$ of points $p$ such that the three conics $C_2 , C_4 , C_6$ mutually intersect at \emph{three} points has dimension at least one.  For each point $p \in U$, let $\varphi (p)$ be the 4th point of intersection of $C_2$ and $C_4$, different from $p$, $p_6$, and $p_7$.  Let us denote the conic $C_6$ by $C_p$, to emphasize its dependence on the point $p$, and consider the incidence correspondence
$$ Z = \{ (\varphi (p), C_p )\} \in U \times \PP H^0 (\mathcal{O}_{\PP^2} (2)) . $$
The fiber over a conic $C_p$ is the set of points $\varphi (p)$ obtained by varying the point $p$ along the fixed conic $C_p$.  If the fiber consists of just one point, it would imply that the conic $C_2$ containing $\{ q_2, p_5, p_6, p_7 \}$ and the fixed point $\varphi (p)$ passes through every point of $C_p$.  This in turn would imply that $C_2$ and $C_p$ are in fact the same conic, which is impossible.  It follows that the closure of the fiber over $C_p$ is a curve.  This curve must intersect $C_p$ in at least one point, and by definition such a point lies in $Y$.  Since a generic fiber, before taking the closure, contains a point of $Y$, we see that $Y$ is at least 1-dimensional.
  
Let $p_4$ be a point in the intersection $Y \cap C_9$, and let $p_8 \neq p_4$ be another point contained in the three conics $C_2 , C_4 , C_6$.  Finally, let $C_8$ be the unique conic containing $\{ p_3, p_4 , p_5, p_6, p \}$.  We then see that the pencil of lines through $p_6$ is numerically equivalent to $F_9$, and the line $L$ with marked points $q_i$ is one of the members of this pencil.

\end{proof}

\begin{corollary}
\label{cor:NEqualsNine}
The divisor $D_{oct}$ spans an extremal ray of $\Eff ( \M_{0,9} )$.  It lies outside the cone generated by boundary and hypertree divisors.\footnote{The divisor $D_{oct}$ also lies outside the cone generated by these and Opie's extremal rays.  By extremality, it suffices to check that $D_{oct}$ does not coincide with any class in the $S_n$ orbit of Opie's class.  The formula in \cite[\S5]{Opi13} shows that the degree of Opie's divisor is greater than 2 for all but possibly the last two psi-classes; since this formula completely describes the divisor class, one can use standard results, such as those in \cite{Morrison}, to compute the degree relative to the last two psi-classes and observe that it is indeed greater than 2 there as well.}
\end{corollary}

\begin{proof}
Since $D_{oct}$ is irreducible, $F_9$ covers $D_{oct}$, and $F_9 \cdot D_{oct} < 0$, we know $D_{oct}$ generates an extremal ray of $\Eff ( \M_{0,9} )$.  By Propositions \ref{Prop:degtwoCT} and \ref{prop:UniCom}, $D_{oct}$ is not equal to a hypertree or boundary divisor, and since it is extremal it therefore cannot be in the cone that they span.
\end{proof}

\subsubsection{The Case of $n=7$}\label{subsec:twotri}

Now consider the divisor $D_{tri} := D_\Delta$ where $\Delta$ is the edges of a disjoint union of two triangles.  We saw in Proposition \ref{Prop:MinimalGraphs} that $D_{tri}$ is effective if and only if $\Char(k) = 2$.  On the other hand, a straightforward linear algebra computation shows that the class $2D_{tri}$ is effective in every characteristic.  Indeed, if $\Delta'$ is the 1-complex depicted in Figure \ref{fig:tribritri}, then by taking a linear combination of the two sections we know to exist in $\HH^0 ( \M_{0,7}, 2D_{\Delta'} )$ by Example \ref{ex:2Dboundary}, one obtains a section divisible by a product of $x_I$ variables, and the quotient is in $\HH^0 ( \M_{0,7}, 2D_{tri} )$.  As in the previous section, our goal is to construct a curve $F_7$ such that $F_7 \cdot D_{tri} < 0$.

Fix 6 general lines $L_1 , \ldots , L_6 \subset \PP^2$, and consider the pencil of cubics spanned by $L_1 + L_2 + L_3$ and $L_4 + L_5 + L_6$.  Since the lines are general, no 3 of the basepoints of this pencil lie on a line other than one of the $L_i$'s, and thus the 6 other singular elements of the pencil are irreducible.  Let $C$ be an irreducible singular cubic in this pencil, and let $p \in C$ be the singular point.

The pencil of lines through $p$ gives a curve in $\M_{0,7}$ as follows:  each line through $p$ is labelled by the intersection of the line with $L_i$, with the $7^{th}$ point labelled by the unique point of intersection of the line with $C$ other than $p$.  More precisely, if we let $S$ be the blow-up of $\PP^2$ at $p$ and the intersection points of the 6 lines, then the strict transforms $\tilde{L_i} , \tilde{C}$ of the lines $L_i$ and the cubic $C$ give disjoint sections of the map $\pi_p : S \to \PP^1$.  We write $F_7 \subset \M_{0,7}$ for the curve obtained in this way.  By construction, we have:
$$ F_7 \cdot H = - \tilde{C}^2 = -(3^2 - 2^2 - 9) = 4 $$
\[F_7 \cdot E_I = \begin{cases}
1 & \textrm{if } I = \{ i,j \} \textrm{ where $i \leq 3$ and $j \geq 4$}\\
0 & \textrm{otherwise.}\\
\end{cases}\]
It follows that
$$ F_7 \cdot D_{tri} = 2 \cdot 4 - 9 = -1. $$

\begin{corollary}
\label{cor:NEqualsSeven}
The class $D_{tri}$ lies outside the cone generated by boundary and hypertree divisors.
\end{corollary}

\begin{proof}
The numerical calculation above shows that $F_7$ has nonnegative intersection with all of the boundary divisors.  Similarly, the intersection of $F_7$ with any hypertree divisor (of which, up to symmetry, there are only two) is positive.  Since $F_7 \cdot D_{tri} < 0$, $D_{tri}$ cannot be an effective linear combination of boundary and hypertree divisors.
\end{proof}

\begin{remark}
We strongly suspect that the curve $F_7$ covers the divisor $2D_{tri}$, and thus $D_{tri}$ spans an extremal ray of $\Eff( \M_{0,7} )$, but at present we have no proof.
\end{remark}

\section{Topics for further study}\label{sec:Future}
\begin{itemize}
\item Extend the classification of minimal balanceable complexes from dimension one (Proposition \ref{Prop:MinimalGraphs}) to higher dimensions (if necessary, restrict to non-singular $d$-complexes with $d \le n-5$).
\item When is the class $D_\Delta$ irreducible if the $d$-complex $\Delta$ is allowed to be singular and/or satisfy $d > n-5$? (cf., Theorem \ref{thm:CoxGen}.)
\item Produce examples of divisor classes whose effectivity requires $\Char(k)=p$ for some $p\ne 2$. (Cf., the example used in Theorem \ref{thm:CharTwo}.)
\item Find a direct bridge between the combinatorics of hypertrees and of a corresponding class of balanceable complexes.  Some hypertree divisors arise from even triangulations of the sphere \cite[\S7]{CT13}; does this relate to our spherical triangulations? (Cf., \S\ref{sec:Deg3}.)
\item Study the relation between a divisor class $D$ and its multiples $mD$ in terms of the associated balanceable complexes.
\item Develop a method for determining when an irreducible divisor class spans an extremal ray of the effective cone, and for proving extremality in such cases. (Cf., \S\ref{subsec:twotri}.)
\end{itemize}


\end{document}